\documentclass[12pt,twoside]{article}
\usepackage{amssymb,amsmath,amsthm,enumerate,geometry,color}
\newtheorem{theorem}{Theorem}%[section]
\newtheorem{lemma}[theorem]{Lemma}
\newtheorem{proposition}[theorem]{Proposition}
\newtheorem{corollary}[theorem]{Corollary}
\newtheorem{definition}[theorem]{Definition}
\newtheorem{remark}[theorem]{Remark}
\newtheorem{example}[theorem]{Example}

\numberwithin{equation}{section}
\providecommand{\keywords}[1]{\textbf{\textit{Keywords:}} #1}
\providecommand{\subjclass}[1]{\textbf{\textit{2010 Mathematics Subject Classification:}} #1}
\begin{document}
\date{}
\title{Beurling densities and frames of exponentials on the union of small balls}
\author{Jean-Pierre Gabardo\thanks{Supported by an NSERC grant}
\\ Department of Mathematics and Statistics\\
 McMaster University\\Hamilton, Ontario, L8S 4K1, Canada\\
gabardo@mcmaster.ca
\and
Chun-Kit Lai\thanks{Supported by the mini-grant of ORSP of San Francisco State University (Grant No: ST659)}\\ Department of Mathematics\\
 San Francisco State University\\
San Francisco, CA 94132, USA\\
cklai@sfsu.edu
}

\maketitle
\begin{abstract}
If $x_1,\dots,x_m$ are finitely many points in $\mathbb{R}^d$,
let $E_\epsilon=\cup_{i=1}^m\,x_i+Q_\epsilon$,
where $Q_\epsilon=\{x\in \mathbb{R}^d,\,\,|x_i|\le \epsilon/2, \, i=1,...,d\}$ and let $\hat f$ denote the
Fourier transform of $f$.
Given a positive Borel measure $\mu$
on $\mathbb{R}^d$, we provide a necessary and sufficient condition for the frame inequalities
$$
A\,\|f\|^2_2\le \int_{\mathbb{R}^d}\,|\hat f(\xi)|^2\,d\mu(\xi)\le B\,\|f\|^2_2,\quad f\in L^2(E_\epsilon),
$$
to hold for some $A,B>0$ and for some $\epsilon>0$ sufficiently small. If $m=1$,
we show that the limits of the optimal lower and upper frame bounds as $\epsilon\rightarrow 0$
are equal, respectively, to the lower and upper Beurling density of $\mu$.
When $m>1$, we extend this result by defining a matrix version of Beurling
density.  Given a (possibly dense)
subgroup $G$ of $\mathbb{R}$, we then consider the problem of characterizing
those measures $\mu$ for which the inequalities above
hold whenever $x_1,\dots,x_m$ are finitely many points in $G$ (with $\epsilon$ depending on those points, but
not $A$ or $B$). We point out an interesting connection between this problem and the notion of
well-distributed sequence when $G=a\,\mathbb{Z}$ for some $a>0$. Finally, we show the existence of a
discrete set $\Lambda$ such that the
measure $\mu=\sum_{\lambda}\,\delta_\lambda$ satisfy the property above for the whole group
$\mathbb{R}$.
\end{abstract}
\keywords{Frame, Beurling density, almost-periodic function.}\\
\subjclass{Primary 42C15; Secondary 42A75.}
\section{Introduction}
If $E \subset \mathbb{R}^d$ is measurable with $|E|<\infty$,
where $|E|$ denotes the Lebesgue measure of $E$,
 let $L^2(E)$ be the space of (complex-valued) square-integrable functions
on $E$.
 Given a discrete subset $\Lambda\subset \mathbb{R}^d$,
consider the collection of exponentials
$\mathcal{E}(\Lambda)=\{e^{2\pi i \lambda\cdot x},\,\, \lambda\in\Lambda \}$.
This collection forms a {\it Fourier frame} for $L^2(E)$ if
there exist constants $A,B>0$  such that
\begin{equation}\label{discrete-frame}
A\,\|f\|_2^2\leq \sum_{\lambda\in\Lambda}
\left|\int_{E}\,f(x)\,e^{-2\pi i \lambda\cdot x}\,dx\right|^2
\leq B\|f\|_2^2,\quad f\in L^2(E).
\end{equation}
If we allow the constant $A$ to be zero as well, $\mathcal{E}(\Lambda)$
is then called a {\it Bessel collection} in  $L^2(E)$.
Defining the Fourier transform of a  function $f\in L^1(\mathbb{R}^d)$ by the formula
$$
 \hat{f}(\xi) = \int_{\mathbb{R}^d}\, f(x)\,e^{-2\pi i \xi\cdot x}\,dx,\quad \xi\in \mathbb{R}^d,
$$
and the measure $\mu=\delta_{\Lambda}:=\sum_{\lambda\in\Lambda}\,\delta_{\lambda}$,
we can rewrite the frame inequalities (\ref{discrete-frame}) as
\begin{equation}\label{measure-frame}
A\,\|f\|_2^2\leq \int_{\mathbb{R}^d}\,|\hat f(\lambda)|^2\,d\mu(\lambda)
\leq B\,\|f\|_2^2,\quad f\in L^2(E).
\end{equation}
If the inequalities (\ref{measure-frame}) hold for a general positive Borel $\mu$ on
$\mathbb{R}^d$, we call the measure $\mu$ an  {\it exponential  frame measure}
(abbr.~$\mathcal{F}$-measure)   for  $L^2(E)$.
Similarly, we call $\mu$ an {\it exponential Bessel measure} (abbr.~$\mathcal{B}$-measure) for $L^2(E)$
if $A$ is allowed to be
$0$ in  (\ref{measure-frame}).
\medskip

The notion of frame was first introduced by Duffin and Schaeffer \cite{[DS]}.
This area of research has been developing rapidly in recent years, 
both in theory and applications, and has become one of the main tools in applied harmonic analysis, 
including Gabor analysis, wavelet theory,  sampling theory and signal processing. 
Readers may refer to \cite{Chr03} for general background on the theory of frames. 
Fourier frames were  first introduced in \cite{[DS]} under the name of 
non-harmonic Fourier series. They are theoretically attractive 
since in contrast to orthonormal bases, 
Fourier frames are easy to construct on bounded sets and are robust to 
small perturbation of the set of frequencies. 
They are also valuable in applications since Fourier frames on 
$L^2(E)$ allow for the reconstruction of signals whose frequency band is 
supported on $E$.  We refer the reader to \cite{Yo} for classical results
concerning frames of exponentials. 
The concept of ${\mathcal F}$-measure as defined in (\ref{measure-frame}) is a particular case
 of ``generalized frame'' associated with a measure introduced in \cite{[GH]}. 
In addition to making our results more general, it allows us to 
simplify notations and provide further 
flexibility when considering problems about Fourier frames \cite{[DHW],[GL1]}.

\medskip

One of the main goal of this paper is to provide necessary
and sufficient conditions for a measure $\mu$ 
(and a discrete set $\Lambda$) to be an $\mathcal{F}$-measure (resp.~a
$\mathcal{B}$--measure)
for $L^2(E)$ when $E$ is a union of finitely many sufficiently ``small'' balls.  
As in the case with many results related to sampling \cite{[GR],[Ja],[Lan]}, 
the notions of upper and lower Beurling density appear naturally in the solution of our 
problems.
%It turns out that these conditions are intimately related to the notion of upper
%and lower Beurling density of the measure $\mu$ in the case of a single ball
%and of a matricial generalization of these notions of densities in the case
%of several balls, as will be made more precise later on.
We first recall that the {\it upper and lower Beurling density} of a
positive Borel measure $\mu$ on $\mathbb{R}^d$ are defined, respectively, as
$$
\mathcal{D}^{+}(\mu) = \limsup_{h\rightarrow\infty}\sup_{x\in \mathbb{R}^d}\frac{\mu(x+Q_{h})}{h^d},
$$
and
$$
\mathcal{D}^{-}(\mu)
= \liminf_{h\rightarrow\infty}\inf_{x\in \mathbb{R}^d}\frac{\mu(x+Q_{h})}{h^d},
$$
where
$$
Q_h=\{x=(x_1,\dots,x_d)\in \mathbb{R}^d,\,\,|x_i|\le d/2,\,\,i=1,\dots,d\}
$$
is the hypercube of side length $h$ centered at the origin. If $\mathcal{D}^{-}(\mu)=\mathcal{D}^{+}(\mu)<\infty$,  the common value of both densities, denoted by $\mathcal{D}(\mu)$, is called the the  Beurling density
of $\mu$. Note that, if $x\in\mathbb{R}^d$ and $A,B$ are subsets of $\mathbb{R}^d$, we use the notation
$x+A$ for the set $\{x+a,\,\,a\in A\}$ and $A+B$ for the set $\{a+b,\,\,a\in A,\,\,b\in B\}$.
If $\Lambda$ is a countable set contained in $\mathbb{R}^d$,
we define $\mathcal{D}^{-}(\Lambda)=\mathcal{D}^{-}(\delta_{\Lambda})$,
 $\mathcal{D}^{+}(\Lambda)=\mathcal{D}^{+}(\delta_{\Lambda})$
and  $\mathcal{D}(\Lambda)=\mathcal{D}(\delta_{\Lambda})$, where
$\delta_{\Lambda}=\sum_{\lambda\in\Lambda}\delta_{\lambda}$.
A positive Borel measure $\mu$ is called {\it translation-bounded} if  there
exists a constant $C>0$ such that
$$
\mu(x+[0,1]^d)\leq C,\quad  \forall x\in\mathbb{R}^d.
$$
It is known that $\mu$ is translation-bounded if and only if ${\mathcal D}^{+}(\mu)<\infty$ (\cite{Ga}, see also Proposition \ref{translation-bounded} in Section 2).

\medskip

Suppose that $B(a,\epsilon)$ is a ball of radius $\epsilon$ centered at $a$. 
It was shown  in \cite[Proposition 2.5]{[Lai]}, 
using a perturbation argument, that if $D^{-}(\Lambda)>0$, 
then for sufficiently small $\epsilon>0$, ${\mathcal E}(\Lambda)$ is a Fourier frame for 
$L^2(B(a,\epsilon))$ for any $a\in{\mathbb R}^d$. The converse clearly holds by the 
density result of Landau (\cite{[Lan]}) (i.e.~if ${\mathcal E}(\Lambda)$ 
is a frame for $L^2(E),$ then ${\mathcal D}^-(\Lambda)\ge |E|$).
 A similar result was also obtained by Beurling (\cite{[Beu]}) who showed that 
if $\Lambda\subset{\mathbb R}^d$ is a uniformly discrete set 
satisfying the covering property 
$$
\bigcup_{\lambda\in\Lambda}(B(0,1/r)+\lambda)={\mathbb R}^d,
$$
then, $\Lambda$ is a Fourier frame for $L^2(\epsilon B(0,r))$ if $\epsilon<1/4$.
(See also \cite{OU1}.) This is now known as the Beurling covering 
theorem and has found application in MRI reconstruction (\cite{[BW]}). 
Our first theorem complements these results by providing a precise relation 
between the Beurling densities and the frame bounds in the case where $E$
is a small neighborhood of a single point in $\mathbb{R}^d$
(which we take to be a cube for convenience).

\begin{theorem}\label{measure-frame1}
Let $\mu$ be positive, locally finite Borel measure on $\mathbb{R}^d$.
Then the following are equivalent.
\begin{enumerate}[(a)]
\item There exists constants $A,B>0$ and $\epsilon>0$
such that
$$
A\,\|f\|_2^2\le \int_{\mathbb{R}^d}\,|\hat f(\xi)|^2\,d\mu(\xi)
\le B\,\|f\|_2^2,\quad f\in L^2(Q_\epsilon).
$$
\item We have $0<\mathcal{D}^-(\mu)
\le \mathcal{D}^+(\mu)<\infty.$
\end{enumerate}
Moreover, if (a) holds, we have $A\le \mathcal{D}^-(\mu)
\le \mathcal{D}^+(\mu)\le B$  and if (b) holds we can find,
for any $\rho>0$ a corresponding $\epsilon>0$ such that the inequalities in (a) hold with
$A=\mathcal{D}^-(\mu)-\rho$
and $B=\mathcal{D}^+(\mu)+\rho$.
\end{theorem}

\begin{remark}\label{rmmark_frame}
Note that the theorem also has a ``Bessel" version, in which $A=0$ and only $D^+(\mu)$ plays a role, (See Theorem \ref{th1.2}).
\end{remark}

This theorem leads to the following corollary, showing that Beurling densities as limit of optimal frame bounds of small ball.

\medskip

\begin{corollary}\label{measure-frame2}
Under the assumptions of Theorem \ref{measure-frame1}, suppose that $\mathcal{D}^{+}(\mu)<\infty$.
Define $A_{\epsilon}$ and $B_{\epsilon}$ to be the optimal bounds for
the inequalities
\begin{equation}
A_{\epsilon}\,\|f\|_2^2\leq \int_{\mathbb{R}^d}|\hat{f}(\lambda)|^2\,d\mu(\lambda)
\leq B_{\epsilon}\, \|f\|_2^2, \ f\in L^2\left(B(a,\epsilon)\right),
\end{equation}
Then, $\lim_{\epsilon\rightarrow 0} A_{\epsilon}=
\mathcal{D}^{-}(\mu)$
and
$\lim_{\epsilon\rightarrow 0} B_{\epsilon}=\mathcal{D}^{+}(\mu)$.
\end{corollary}
 (Theorem \ref{measure-frame1}  will not be proved here
as they are  particular cases of Theorem \ref{th1.2} and Corollary \ref{cor2.2}, respectively,
which will be proved in section 2.)

\medskip

If $\Omega = \bigcup_{i=1}^{N}(a_i+Q_{\epsilon})$ is a 
finite union of disjoint cubes with side length $\epsilon>0$, 
Theorem \ref{measure-frame1} is no longer true since
the lower-frame bound inequality might fail under
the conditions $0<\mathcal{D}^-(\mu)
\le \mathcal{D}^+(\mu)<\infty$ even if $\epsilon$ is small.
For example,  $\mu=\sum_{n\in \mathbb{Z} }\,\delta_n$  is not
an $\mathcal{F}$-measure for the set $\Omega=[-\epsilon,\epsilon]\cup [1-\epsilon,1+\epsilon]$.
To see this, we consider $f = \chi_{[-\epsilon,\epsilon]}-\chi_{[1-\epsilon,1+\epsilon]}$. Then, for all $n
\in\mathbb{Z}$,
we have
$$
\hat{f}(n) = \int_{-\epsilon}^{\epsilon}\,e^{-2\pi i nx}\,dx-e^{2\pi i n}\,\int_{-\epsilon}^{\epsilon}\,e^{-2\pi i nx}\,dx=0.
$$
This means that the collection $\{e^{-2\pi i nx}\}_{n\in \mathbb{Z}}$ is not even complete
in $L^2(\Omega)$.  By introducing notions
of lower and upper density for Borel measures on $\mathbb{R}^d$
taking values in the  cone of positive-definite
matrices, we characterize
$\mathcal{B}$-measures and $\mathcal{F}$-measures
for $L^2(\Omega)$ in the case where $\Omega= \bigcup_{i=1}^{N}(a_i+Q_{\epsilon})$
and $\epsilon$ is small enough. This provides thus analogues of 
Theorem  \ref{measure-frame1} and Corollary \ref{measure-frame2}
for this more general situation. %This raises an interesting question: given a translation bounded, positive
%Borel measure $\mu$ on  $\mathbb{R}^d$ and finitely many points $x_i\in\mathbb{R}^d$,
%under what conditions is $\mu$ an $\mathcal{F}$-measure
%for $L^2(\Omega)$ where $\Omega=\cup_i B(x_i,\epsilon)$ if $\epsilon>0$ is small enough?
%Furthermore, is there a way to compute the limiting values of the best Bessel and/or frame bounds
%as $\epsilon$ goes to zero in analogy to the statements in  Theorems \ref{bessel1} and  \ref{measure-frame1}?

\medskip

After establishing these results, we will 
consider a related problem which involves a uniformity condition
on the frame bounds with respect to a subgroup $G$ of  ${\mathbb R}$.

\begin{definition}\label{def3} {\rm Let $\mu$ positive Borel measure on $\mathbb{R}$ and let
$G$ be a subgroup of $\mathbb{R}$. If $A,B>0$, we say that $\mu$ is a {\it uniform $\mathcal{F}$-measure}
for $G$ with limiting lower bound larger than or equal to $A$ and limiting upper frame bound
less than or equal to  $B$, if  given any $x_1,\dots ,x_M\in G$ and any $\delta>0$,
there exists $\epsilon>0$ such that}
 \begin{equation}\label{eq3.1}
(A-\delta)\,\|f\|_2^2\leq \int_{\mathbb{R}^d}|\hat{f}(\lambda)|^2\,d\mu(\lambda)
\leq (B+\delta)\, \|f\|_2^2, \ f\in L^2\left(\Omega\right),
\end{equation}
{\rm for $\Omega=\bigcup_{j=1}^{N}(x_j+Q_{\epsilon})$.
We denote the collection of such measures by $\mathcal{F}(G,A,B)$.
The notion of {\it uniform $\mathcal{B}$-measure} for $G$ with limiting upper frame bound
less than or equal to  $B$ is defined in a similar way and the collection of such
measures is denoted by $\mathcal{B}(G,B)$. Finally, the measures in the
collection  $\mathcal{F}(G,A,A)$ are called  {\it uniform tight $\mathcal{F}$-measures} with limiting tight
frame bound $A$
for $G$.}
\end{definition}

\medskip
The construction of Borel measures $\mu$ in  $\mathcal{F}(G,A,B)$
 can be viewed as a continuous version of a {\it compressed sensing} problem (See \cite{[FR]} for details about compressed sensing).
A vector $v$ is a finite-dimensional space is $s$-sparse (where $s\ge 1$ is an integer)
 if it has at most $s$ non-zero components. In general, the indices corresponding to the 
non-zero components of $v$ are unknown.
In its discrete and finite-dimensional setting, the compressed sensing problem
consist in trying to recover an $s$-sparse vector $v$  in $\mathbb{C}^d$
by computing the inner products $\langle v, u_i\rangle$ with
some fixed  vectors $u_i$, $i\in I$. As can be expected, the smaller $s$ is, the fewer
vectors $u_i$ are needed for the recovery of the data. If we think of a function
$f$ in $L^2(E)$ as a vector with non-zero components concentrated on the set $E$
and if $\mu=\delta_{\Lambda}$, for some discrete set $\Lambda$,
the fact that $\mu\in \mathcal{F}(G,A,B)$ allows for the recovery of any such 
function from the knowledge of the
inner products $\langle f, e_\lambda\rangle_{L^2(E)}$ 
where  $e_\lambda(x)=e^{-2\pi i \lambda\cdot x}\,\chi_E(x)$.
This will be possible if $f$ is sparse enough in the sense that it should be supported
in  a small enough neighborhood of a finite subset of the group $G$.
The fact that the constants $A,B$ are independent of the points chosen
in the group $G$ implies that the robustness of the reconstruction
formula is also independent of the exact location of this neighborhood.

\medskip
One trivial element inside $ \mathcal{F}(G,1,1)$ for any subgroup $G$ is the Lebesgue measure on ${\mathbb R}$. We are particularly interested in the existence of discrete measures inside these collections.  We first consider the problem with $G$ being a finitely-generated subgroup of ${\mathbb R}$ and completely solve this problem with the help of Theorem \ref{measure-frame1}.
It turns out that, interestingly, in the case of measures of the form $\mu=\delta_\Lambda$,
$\Lambda\subset \mathbb{R}$, the answer to these questions is  related to the probabilistic notion
of ``equidistributed sequence''  or,  more specifically, that of``well-distributed sequence'' 
(Theorem \ref{frame1}). We will show, in particular, that 
{\it if $G = a{\mathbb Z}$, $\delta_{\Lambda}\in{\mathcal F}(G,A,A)$ if and only if 
${\mathcal D}(\Lambda) = A$ and 
$\Lambda$ is a well-distributed sequence (mod $a^{-1}$)} (Corollary \ref{well-dist}). 
Our results can also be interpreted as characterizations of certain inequalities
satisfied by almost-periodic functions with spectrum in the group $G$ (Theorem \ref{a-p}).

\medskip

Finally, we consider the problem for the whole group ${\mathbb R}$. Using a recent result of 
S.~Nitzan, A.~Olevskii and A.~Ulanovskii (\cite{NOU}) about the existence of Fourier frames on 
any unbounded set of finite measure, which is 
based on the solution of the Kadison-Singer problem, 
we deduce {\it the existence of a discrete $\Lambda$ such that 
$\delta_{\Lambda}\in{\mathcal F}({\mathbb R}, A,B)$ for some $A,B$ ($A<B$)}. 
It would be reasonable to think that the measure associated with  a 
simple quasicrystal in the sense of Meyer (\cite{Me1})
may belong to some space ${\mathcal F}({\mathbb R}, A,B)$ in view of 
the results on universal sampling obtained in \cite{MM}, 
but we will show that this is never the case. Since the solution to 
the celebrated Kadison-Singer conjecture in \cite{MSS}  
is a probabilistic result,
it would be interesting, in line with the current research, 
to find some deterministic discrete sets with associated measure belonging to some space
 ${\mathcal F}({\mathbb R}, A,B)$.

\medskip

We organize our paper as follows: in Section 2, we provide the basic preliminary results on Beurling density and introduce the Beurling densities of  Borel measures on $\mathbb{R}^d$
taking values in the  cone of positive-definite
matrices. We will prove the matrix version of Theorem \ref{measure-frame1} and Corollary \ref{measure-frame2} in Section 3. In Section 4, we characterize the measures in ${\mathcal F}(G,A,B)$ and ${\mathcal B}(G,B)$. We study the case $G = {\mathbb R}$ in the last section. %Note that our result certainly holds when small open balls is replaced by any scaled open sets, for convenience and technical
%reasons, we will replace our small balls in $\mathbb{R}^d$ by small hypercubes in our analysis.

\medskip
%It is in fact  possible to make the best constant $B$ as close as we want to  $\mathcal{D}^+(\mu)$
%by choosing $E$ to be ball with a radius small enough
%as the following result shows.
%\begin{theorem}\label{bessel1}
%Let $\mu$ be positive Borel measure on $\mathbb{R}^d$
%and let  $B(a, \epsilon)=\{x\in \mathbb{R}^d,\,\,|x-a|< \epsilon\}$ denote the open
%ball of radius $\epsilon$ centered at $a\in \mathbb{R}^d$.
%Then the following are equivalent.
%\begin{ enumerate}[(a)]
%\item There exist constants $B>0$ and $\epsilon>0$
%such that
%$$
% \int_{\mathbb{R}^d}\,|\hat f(\xi)|^2\,d\mu(\xi)
%\le B\,\|f\|_2^2,\quad f\in L^2(B(a, \epsilon)).
%$$
%\item $ \mathcal{D}^+(\mu)<\infty$ (or, equivalently, $\mu$ is translation bounded).
%\end{ enumerate}
%Moreover, if (a) holds, we have $\mathcal{D}^+(\mu)\le B$ and,   if (b) holds,  we can find,
%for any $\rho>0$, a corresponding $\epsilon>0$ such that the
%inequality in (a) hold with $B=D^+(\mu)+\rho$.
%\end{theorem}
\noindent{\bf Local square integrability of the Fourier transform.} 

\medskip
Before we develop our theory in the next section, we mention an additional consequence of Theorem \ref{measure-frame1}. Note that by the implication (b) $\Longrightarrow$ (a) in Theorem \ref{measure-frame1}, if $\mu$ is translation-bounded (i.e. ${\mathcal D}^{+}(\mu)<\infty$), then the ``analysis'' operator
$$
T:L^2(B(a, \epsilon))\to L^2(\mu):f\mapsto \hat f
$$
is bounded for all $a\in {\mathbb R}^d$ and $\epsilon>0$ small enough, where $B(a,\epsilon)$ is the Euclidean ball of radius $\epsilon$ centered at $a$. and thus so is its adjoint, the ``synthesis'' operator
$T^*: L^2(\mu) \to L^2(B(a, \epsilon))$.
It is easy to see that if $\mu$ is translation bounded,
and $F\in L^2(\mu)$, then $F d\mu$ defines a tempered distribution on $\mathbb{R}^d$.
In that case, taking $g\in C^\infty_0(B(a,\epsilon))$,
we have
$$
( T^*F, g)_2=( F,T g)_{L^2(\mu)}=
\int_{\mathbb{R}^d} F(\lambda)\, \overline{\hat g(\lambda)}\,d \mu(\lambda)
=\langle \mathcal{F}^{-1}(F\,d\mu), \overline{ g} \rangle,
$$
where the bracket $\langle\cdot,\cdot \rangle$ represents the duality
between tempered  distributions in $\mathcal{S}'(\mathbb{R}^d)$
and the test functions in the Schwartz space  $\mathcal{S}'(\mathbb{R}^d)$
and where
$\mathcal{F}^{-1}$ denote the (distributional)
inverse Fourier transform.
It follows that, if the positive Borel measure $\mu$ is translation bounded
and $F\in L^2(\mu)$, then the synthesis operator $T^*$ is defined by
\begin{equation}\label{adjoint}
T^*: L^2(\mu) \to L^2(B(a, \epsilon)): F \mapsto \left.\mathcal{F}^{-1}(F\,d\mu)\right|_{B(a,\epsilon)}.
\end{equation}
In particular, this implies that the distribution defined by $\mathcal{F}^{-1}(F\,d\mu)$ is
locally square-integrable on $\mathbb{R}^d$. This property was actually proved in 1990 by R. Strichartz (\cite[Lemma 4.2]{Str}) using a different method. We now show that the converse of this statement is also true: if $\mu$
is a positive tempered measure on $\mathbb{R}^d$ and
$\mathcal{F}^{-1}(F\,d\mu)$ is
locally square-integrable on $\mathbb{R}^d$ for every $F\in  L^2(\mu)$,
the $\mu$ must be translation-bounded. Indeed it is easily checked that the mapping
defined in (\ref{adjoint}) is closed and it is thus bounded by the closed graph theorem.
 Since the boundedness of
$T^*$ is equivalent to the boundedness of $T$, it follows that  $\mu$ is translation bounded
using the implication (a) $\Longrightarrow$ (b) in Theorem \ref{measure-frame1}. 
We note also, that since the property of
being square-integrable on every ball
$B(a,\epsilon)$, with center $a\in \mathbb{R}^d$ and fixed radius $\epsilon>0$, is clearly equivalent
to being
square-integrable on every ball
$B(a,r)$, where $r>0$ is arbitrary, it follows that if $\mu$ is a $\mathcal{B}$-measure for
$L^2(B(a,\epsilon))$ for some $a\in \mathbb{R}^d$ and some $\epsilon>0$, then
 $\mu$ is a $\mathcal{B}$-measure for $L^2(B(x,r))$, for any $x\in \mathbb{R}^d$ and any $r>0$
(but with the Bessel constant dependent on $r$). We summarize these conclusions in 
the following theorem.

\begin{theorem}\label{th0.3}
Let $\mu$
be a positive tempered measure on $\mathbb{R}^d$. Then $\mathcal{F}^{-1}(F\,d\mu)$ is
locally square-integrable on $\mathbb{R}^d$ for every $F\in  L^2(\mu)$ if and only if  $\mu$ must be translation-bounded.
\end{theorem}

%
%
%We have also a characterization of $\mathcal{F}$-measures for small balls.
%\begin{theorem}\label{measure-frame1}
%Let $\mu$ be positive, locally finite Borel measure on $\mathbb{R}^d$.
%Then the following are equivalent.
%\begin{enumerate}[(a)]
%\item There exists constants $A,B>0$ and $\epsilon>0$
%such that
%$$
%A\,\|f\|_2^2\le \int_{\mathbb{R}^d}\,|\hat f(\xi)|^2\,d\mu(\xi)
%\le B\,\|f\|_2^2,\quad f\in L^2(B_\epsilon).
%$$
%\item We have $0<\mathcal{D}^-(\mu)
%\le \mathcal{D}^+(\mu)<\infty.$
%\end{enumerate}
%Moreover, if (a) holds, we have $A\le \mathcal{D}^-(\mu)
%\le \mathcal{D}^+(\mu)\le B$  and if (b) holds we can find,
%for any $\rho>0$ a corresponding $\epsilon>0$ such that the inequalities in (a) hold with
%$A=\mathcal{D}^-(\mu)-\rho$
%and $B=\mathcal{D}^+(\mu)+\rho$.
%\end{theorem}
%(Note that Theorem \ref{bessel1} and Theorem \ref{measure-frame1} will not be proved here
%as they
%are  particular cases of Theorem \ref{th1.2} and Theorem \ref{th1.1}, respectively,
%which will be proved in section 2.)
%
More generally, if $E=\Omega$ where $\Omega$ is an open subset of
$\mathbb{R}^d$ with $|\Omega|<\infty$, the Bessel inequality
$$
 \int_{\mathbb{R}^d}\,|\hat f(\xi)|^2\,d\mu(\xi)
\le B\,\|f\|_2^2,\quad f\in L^2(\Omega).
$$
is equivalent to the  property
$$
\mathcal{F}^{-1}(F\,d\mu)|_\Omega \in
L^2(\Omega),\quad F\in L^2(\mu).
$$
If $\Omega$ is bounded, this property will hold if and only if $\mu$ is translation bounded
by Theorem \ref{th0.3}.
On the other hand, if $\Omega$ is unbounded but with finite Lebesgue measure,
the Bessel inequality might fail for translation bounded
measures.
For example, if $\mu=\delta_{\mathbb{Z}}=\sum_{n\in \mathbb{Z} }\,\delta_n$,
and, if $F\in L^2(\mu)$, we have
$$
F\,d\mu
=\sum_{n\in \mathbb{Z} }\,c_n\,\delta_n\quad \text{with}\,\,\sum_{n\in \mathbb{Z} }\,|c_n|^2<\infty.
$$
It follows that the set $\{\mathcal{F}^{-1}(F\,d\mu),\,\,F\in L^2(\mu)\}$
is exactly the collection of  locally  square-integrable  $1$-periodic function
on the real line.
If we take
$$
\Omega=\cup_{j\ge 1}\,(j,j+1/j^2)
$$
and define $H(x)$ to be the $1$-periodic function
with
$$
H(x)=x^{-\alpha},\quad 0<x\le 1,
$$
it is easily checked that $H$ is locally square-integrable if and only if $\alpha<1/2$.
In that case, we have thus $H=\mathcal{F}^{-1}(F\,d\mu)$ for some  $F\in L^2(\mu)$.
However, the restriction of $H$ to $\Omega$ is square-integrable
if and only if
$$
\sum_{j=1}^\infty\,\int_{0}^{1/j^2}\,x^{-2\alpha}\,dx
=\frac{1}{1-2\alpha }\,\sum_{j=1}^\infty\,\frac{1}{j^{2-4\alpha}}<\infty
$$
i.e. $\alpha<1/4$. Hence, if we take $\alpha$ with $1/4\le \alpha<1/2$,
$\mathcal{F}^{-1}(F\,d\mu)|_\Omega \not\in
L^2(\Omega)$ and the Bessel property fails.

\section{Densities of  positive matrix-valued measures}

We start by mentioning some properties equivalent to ``translation-boundedness".
\begin{proposition}[\cite{Ga}]\label{translation-bounded} Let $\mu$ be a positive Borel measure on $\mathbb{R}^d$.
Then, the following are equivalent:
\begin{enumerate}[(a)]
\item $\mu$ is translation bounded.
\item  $\mathcal{D}^{+}(\mu)<\infty$.
\item There exists $f\in L^1(\mathbb{R}^d)$  with $f\ge 0$, $\int f\,dx=1$
and a constant $C>0$ such that $\mu\ast f\leq C$ a.e.~on $\mathbb{R}^d$.
\end{enumerate}
\end{proposition}
As the last condition in the previous proposition shows, the notion
of upper Beurling density is related to certain convolution inequalities satisfied by the measure
$\mu$. More generally, we have the following result, which will be useful later on.
\begin{theorem}[\cite{Ga}]\label{Ga-th1} Let $\mu$  be a positive Borel measure on $\mathbb{R}^d$
and let $h\in L^1(\mathbb{R}^d)$ with $h\ge 0$.
Let $A,B>0$ be constants.  Then
\begin{enumerate}[(a)]
\item If $\mu\ast h\leq B$ a.e.~on $\mathbb{R}^d$, then
$
\mathcal{D}^{+}(\mu)\,\int\, h \,dx\leq B.
$
\item If $\mu$ is translation-bounded and $A\leq \mu\ast h$ a.e.~on $\mathbb{R}^d$, then
$
A\leq  \mathcal{D}^{-}(\mu)\,\int\, h \,dx\
$
\end{enumerate}

\end{theorem}
If we assume now that $\mu$ is an $\mathcal{F}$-measure for $L^2(E)$, then applying the frame
inequalities (\ref{measure-frame}) to the function $g(x)=e^{2\pi i \xi\cdot x}\,\overline{f(x)}$,
where $\xi\in\mathbb{R}^d$ and $f\in L^2(E)$, we obtain that
$$
A\,\|f\|_2^2\leq \int_{\mathbb{R}^d}\,|\hat f(\xi-\lambda)|^2\,d\mu(\lambda)
\leq B\,\|f\|_2^2,\quad f\in L^2(E).
$$
which can also be written as
$$
A\,\|f\|_2^2\leq \left(\mu\ast |\hat f|^2\right)(\xi)
\leq B\,\|f\|_2^2,\quad \xi\in\mathbb{R}^d,\,\,f\in L^2(E).
$$
Since $\int_{\mathbb{R}^d}\,|\hat f|^2\,d\lambda=\|f\|_2^2$
by Plancherel's theorem, we can apply Theorem \ref{Ga-th1}
to the function $h:=|\hat f|^2$ to obtain
that
$$
0<A\le \mathcal{D}^{-}(\mu)\le \mathcal{D}^+(\mu)\leq B<\infty.
$$
Of course, the same argument show that, if $\mu$ is a $\mathcal{B}$-measure for
$L^2(E)$ with Bessel constant $B$, then
$$
\mathcal{D}^+(\mu)\leq B<\infty.
$$
This gives a proof for one of the implications in Theorem \ref{measure-frame1}.

We now define appropriate notions of densities for positive matrix-valued measures
generalizing the known notions of Beurling densities defined in the introduction.
\begin{definition} We will denote by $\mathcal{MP}_N( \mathbb{R}^d)$ the set of $N\times N$
matrices $\check \mu=(\mu_{i,j})$ whose entries $\mu_{i,j}$, $1\le i,j\le N$,
 are complex, locally finite  Borel measures
on $\mathbb{R}^d$
and are positive-definite in the sense that, for any $\mathbf{v}=(v_1,\dots,v_N)\in  \mathbb{C}^N$, we have
\begin{equation}\label{muv}
\check \mu_{\mathbf{v}}:=\sum_{i,j=1}^N\,\mu_{i,j}\,v_i\overline{v_j} \ge 0,
\end{equation}
i.e.~the left-hand side of the previous inequality defines a positive measure on $\mathbb{R}^d$.
\end{definition}

The lower and upper Beurling densities of an element $\check \mu$ of $\mathcal{MP}_N( \mathbb{R}^d)$
can then be defined respectively as
$$
\mathcal{D}^{-}_N(\check{\mu}):=\inf\left\{\mathcal{D}^{-}(\check \mu_{\mathbf{v}}):
\mathbf{v}\in  \mathbb{C}^N,\,\,\|\mathbf{v}\|_2=1\right\}
$$
and
$$
\mathcal{D}^{+}_N(\check{\mu}):=\sup\left\{\mathcal{D}^{+}(\check \mu_{\mathbf{v}}):
\mathbf{v}\in  \mathbb{C}^N,\,\, \|\mathbf{v}\|_2=1\right\}.
$$
If $E$ is a bounded Borel subset of $\mathbb{R}^d$ and  $\check \mu\in\mathcal{MP}_N( \mathbb{R}^d)$,
the $N\times N$ matrix $\check \mu(E)$ with (complex) entries  $\mu_{i,j}(E)$ is positive-definite
by definition. Its eigenvalues are thus real and non-negative and we can define
$\lambda_{\text{max}}(\check{\mu},E)$ and
$\lambda_{\text{min}}(\check{\mu},E)$, to be the largest and smallest eigenvalue of
$\check \mu(E)$, respectively.
The following lemma provides an alternative definition of the
lower and upper Beurling densities for elements of  $\mathcal{MP}_N( \mathbb{R}^d)$.
\begin{lemma}\label{altdef}
Let $\check \mu \in \mathcal{MP}_N( \mathbb{R}^d)$.
Then, we have
$$
\mathcal{D}^{+}_N(\check{\mu})=
\limsup_{h\rightarrow\infty}\,\sup_{x\in \mathbb{R}^d}\,
\frac{\lambda_{\text{max}}(\check{\mu},x+Q_{h})}{h^d}
$$
and, if $\mathcal{D}^{+}_N(\check{\mu})<\infty$, we have also
$$
\mathcal{D}^{-}_N(\check{\mu})=
\liminf_{h\rightarrow\infty}\,\inf_{x\in \mathbb{R}^d}\,
\frac{\lambda_{\text{min}}(\check{\mu},x+Q_{h})}{h^d}.
$$
\end{lemma}
\begin{proof} Using standard properties of positive-definite matrices, we have, for any
bounded Borel subset $E$ of $\mathbb{R}^d$ and any  $\mathbf{v}\in  \mathbb{C}^N$, that
$$
\lambda_{\text{min}}(\check{\mu},E)\,\|\mathbf{v}\|_2^2
\le \sum_{i,j=1}^N\,\mu_{i,j}(E)\,v_i\overline{v_j}
\le \lambda_{\text{max}}(\check{\mu},E)\,\|\mathbf{v}\|_2^2.
$$
In particular, if $\|{\bf v}\|_2=1$, we have
$$
\mathcal{D}^{+}(\check \mu_{\mathbf{v}})=
\limsup_{h\rightarrow\infty}\,\sup_{x\in \mathbb{R}^d}\,\frac{\check\mu_{\mathbf{v}}(x+Q_{h})}{h^d}
\le \limsup_{h\rightarrow\infty}\,\sup_{x\in \mathbb{R}^d}\,
\frac{\lambda_{\text{max}}(\check{\mu},x+Q_{h})}{h^d}
$$
and
$$
\mathcal{D}^{-}(\check \mu_{\mathbf{v}})=
\liminf_{h\rightarrow\infty}\,\inf_{x\in \mathbb{R}^d}\,\frac{\check\mu_{\mathbf{v}}(x+Q_{h})}{h^d}
\ge \liminf_{h\rightarrow\infty}\,\inf_{x\in \mathbb{R}^d}\,
\frac{\lambda_{\text{min}}(\check{\mu},x+Q_{h})}{h^d}
$$
which show that
$$
\mathcal{D}^{+}_N(\check{\mu})\le
\limsup_{h\rightarrow\infty}\,\sup_{x\in \mathbb{R}^d}\,
\frac{\lambda_{\text{max}}(\check{\mu},x+Q_{h})}{h^d}
\quad \text{and} \quad
\mathcal{D}^{-}_N(\check{\mu})\ge
\liminf_{h\rightarrow\infty}\,\inf_{x\in \mathbb{R}^d}\,
\frac{\lambda_{\text{min}}(\check{\mu},x+Q_{h})}{h^d}.
$$
To prove the reverse inequality
\begin{equation}\label{upper-ineq}
\limsup_{h\rightarrow\infty}\,\sup_{x\in \mathbb{R}^d}\,
\frac{\lambda_{\text{max}}(\check{\mu},x+Q_{h})}{h^d}\le \mathcal{D}^{+}_N(\check{\mu}),
\end{equation}
we can assume that $\mathcal{D}^{+}_N(\check{\mu})<\infty$.
 Consider sequences $\{h_n\}$ and $\{x_n\}$, with $h_n>0$,
$x_n\in \mathbb{R}^d$ and $h_n\to \infty$ as $n\to \infty$, such that
$$
\lim_{n\to \infty}\,
\frac{\lambda_{\text{max}}(\check{\mu},x_n+Q_{h_n})}{h_n^d}=
\limsup_{h\rightarrow\infty}\,\sup_{x\in \mathbb{R}^d}\,
\frac{\lambda_{\text{max}}(\check{\mu},x+Q_{h})}{h^d}.
$$
Let $ \mathbf{v}_n\in  \mathbb{C}^N$ be an eigenvector of norm $1$ associated with the largest
eigenvalue of the matrix $\check{\mu}(x_n+Q_{h_n})$.
We have then
$$
\lambda_{\text{max}}(\check{\mu},x_n+Q_{h_n})=\check\mu_{\mathbf{v}_n}(x_n+Q_{h_n}).
$$
Letting $ \mathbf{e}_i$, $i=1,\dots,N$, denote the vectors in the standard basis of
$\mathbb{C}^N$, it follows that $\check \mu_{\mathbf{e}_i}=\mu_{ii}$, $i=1,\dots,N$,
and, in particular, $\mathcal{D}^+(\mu_{ii})<\infty$ since we assume that
 $\mathcal{D}^{+}_N(\check{\mu})<\infty$. We can thus find constants $K,h_0>0$ such that
$$
\frac{\mu_{ii}(x+Q_{h})}{h^d}\le K,\quad 1\le i\le N,\,\,x\in \mathbb{R}^d,\,\,h\ge h_0.
$$
Using another well-known property of positive-definite matrices, we have also
that
$$
\frac{\left|\mu_{ij}(x+Q_{h})\right|}{h^d}\le
\left(\frac{\mu_{ii}(x+Q_{h})}{h^d}\right)^{1/2}\,
\left(\frac{\mu_{jj}(x+Q_{h})}{h^d}\right)^{1/2}\le K,\quad 1\le i,j\le N,\,\,x\in \mathbb{R}^d,
$$
for $h\ge h_0$. The entries of the matrices $G_n:=\check{\mu}(x_n+Q_{h_n})/h_n^d$ are thus uniformly
bounded, and we can assume, by compactness, after passing to a subsequence if necessary that
$$
 G_n\to G\quad \text{and}\quad \mathbf{v}_n \to \mathbf{v},\quad n\to \infty,
$$
where $G$ is a positive-definite $N\times N$ matrix and $\mathbf{v}$ a unit vector in $\mathbb{C}^N$.
We have then,
\begin{align*}
\mathcal{D}^{+}_N(\check{\mu})\ge &\lim_{n\to \infty}\,\check{\mu}_{\mathbf{v}}(x_n+Q_{h_n})/h_n^d=
\lim_{n\to \infty}\,\langle G_n\mathbf{v},\mathbf{v}\rangle=
\lim_{n\to \infty}\,\langle G_n\mathbf{v_n},\mathbf{v_n}\rangle\\
&=\lim_{n\to \infty}\,\check{\mu}_{\mathbf{v}_n}(x_n+Q_{h_n})/h_n^d
=\lim_{n\to \infty}\,\lambda_{\text{max}}(\check{\mu},x_n+Q_{h_n})/h_n^d\\
&=\limsup_{h\rightarrow\infty}\,\sup_{x\in \mathbb{R}^d}\,
\frac{\lambda_{\text{max}}(\check{\mu},x+Q_{h})}{h^d},
\end{align*}
which proves the inequality (\ref{upper-ineq}).
It remains to prove the inequality
\begin{equation}\label{lower-ineq}
\liminf_{h\rightarrow\infty}\,\inf_{x\in \mathbb{R}^d}\,
\frac{\lambda_{\text{min}}(\check{\mu},x+Q_{h})}{h^d}\ge \mathcal{D}^{-}_N(\check{\mu})
\end{equation}
under the additional assumption that $\mathcal{D}^{+}_N(\check{\mu})<\infty$.
Consider sequences $\{k_n\}$ and $\{y_n\}$, with $k_n>0$,
$y_n\in \mathbb{R}^d$ and $k_n\to \infty$ as $n\to \infty$, such that
$$
\lim_{n\to \infty}\,
\frac{\lambda_{\text{min}}(\check{\mu},y_n+Q_{k_n})}{k_n^d}=
\liminf_{h\rightarrow\infty}\,\inf_{x\in \mathbb{R}^d}\,
\frac{\lambda_{\text{min}}(\check{\mu},x+Q_{h})}{h^d}.
$$
Let $ \mathbf{u}_n\in  \mathbb{C}^N$ be an eigenvector of norm $1$ associated with the smallest
eigenvalue of the matrix $\check{\mu}(y_n+Q_{k_n})$.
Since  $\mathcal{D}^{+}_N(\check{\mu})<\infty$, the entries of the matrix
$H_n:=\check{\mu}(y_n+Q_{k_n})/k_n^d$ are uniformly
bounded and, similarly, as above we can assume that
$$
 H_n\to H\quad \text{and}\quad \mathbf{u}_n \to \mathbf{u},\quad n\to \infty,
$$
where $H$ is a positive-definite $N\times N$ matrix and $\mathbf{u}$ a unit vector in $\mathbb{C}^N$.
We have then,
\begin{align*}
\mathcal{D}^{-}_N(\check{\mu})\le &\lim_{n\to \infty}\,\check{\mu}_{\mathbf{u}}(y_n+Q_{k_n})/k_n^d=
\lim_{n\to \infty}\,\langle H_n\mathbf{u},\mathbf{u}\rangle=
\lim_{n\to \infty}\,\langle H_n\mathbf{u_n},\mathbf{u_n}\rangle\\
&=\lim_{n\to \infty}\,\check{\mu}_{\mathbf{u}_n}(y_n+Q_{k_n})/k_n^d
=\lim_{n\to \infty}\,\lambda_{\text{min}}(\check{\mu},y_n+Q_{k_n})/k_n^d\\
&=\liminf_{h\rightarrow\infty}\,\inf_{x\in \mathbb{R}^d}\,
\frac{\lambda_{\text{min}}(\check{\mu},x+Q_{h})}{h^d},
\end{align*}
proving our claim.
\end{proof}

\section{Fourier Frames on the union of small cubes}

 In this section, we will be exclusively interested in elements $\check \mu$
of $\mathcal{MP}_N( \mathbb{R}^d)$ constructed  starting
from a positive, locally finite Borel measure $\mu$ on  $\mathbb{R}^d$ and $N$ points
$x_1,\dots, x_N\in \mathbb{R}^d$. We  define $\check \mu=(\mu_{ij})$ by the formula
\begin{equation}\label{matrix-def}
d\mu_{ij}(\xi)=e^{-2\pi i (x_i-x_j)\cdot \xi}\,d\mu(\xi),\quad 1\le i,j\le N.
\end{equation}

If $E$ is a bounded subset of $\mathbb{R}^d$, the associated  matrix $(\mu_{ij}(E))$
has thus entries
$$
\mu_{ij}(E)=\int_E\,e^{-2\pi i (x_i-x_j)\cdot \xi}\,d\mu(\xi),
\quad i,j=1,\dots N.
$$
Furthermore, if $\mathbf{v}$ a  vector in $\mathbb{C}^N$, we have
\begin{equation}\label{eq3.2}
d\check \mu_{\mathbf{v}}=\sum_{i,j=1}^N\,d\mu_{i,j}\,v_i\,\overline{v_j}=
\big|\sum_{j=1}^N\,v_j\,e^{-2\pi i x_j\cdot \xi}\big|^2 d\mu(\xi),
\end{equation}
showing that $\check \mu$ is a positive matrix-valued measure.
\begin{example}
{\rm
If $N=2$, we  can use Lemma \ref{altdef} to obtain an explicit formula for the densities 
$\mathcal{D}^{-}_2(\check{\mu})$ and $\mathcal{D}^{+}_2(\check{\mu})$,
defined at the end of the previous section, which are
associated with a measure $\mu$ and two distinct points $x_1,x_2\in \mathbb{R}^d$
using (\ref{matrix-def}).
For any bounded Borel set $E\subset \mathbb{R}^d$, we have
\begin{align*} \displaystyle
\check{\mu}(E)&=
\begin{bmatrix}
\int_E\,1\,d\mu(\xi) &\int_E\,e^{-2\pi i (x_1-x_2)\cdot \xi}\,d\mu(\xi)\\
\, & \,\\
\int_E\,e^{-2\pi i (x_2-x_1)\cdot \xi}\,d\mu(\xi)&
\int_E\,1\,d\mu(\xi) 
\end{bmatrix}\\
\end{align*} 
and the  eigenvalues of $\check{\mu}(E)$ are given by
$$\lambda_{\text{max}}(\check\mu,E)=
\int_E\,1\,d\mu(\xi) +\left|\int_E\,e^{-2\pi i (x_1-x_2)\cdot \xi}\,d\mu(\xi)\right|
$$
and
$$
\lambda_{\text{min}}(\check\mu,E)
=\int_E\,1\,d\mu(\xi) -\left|\int_E\,e^{-2\pi i (x_1-x_2)\cdot \xi}\,d\mu(\xi)\right|.
$$
The  densities are then computed as
$$
\mathcal{D}^{-}_2(\check{\mu})
=\liminf_{R\to \infty}\,\inf_{t\in \mathbb{R}^d}\,\frac{1}{R^d}\,
\left\{\int_{t+Q_R}\,1\,d\mu(\xi) -\left|\int_{t+Q_R}\,e^{-2\pi i (x_1-x_2)\cdot \xi}\,d\mu(\xi)\right|\right\}.
$$
and
$$
\mathcal{D}^{+}_2(\check{\mu})
=\limsup_{R\to \infty}\,\sup_{t\in \mathbb{R}^d}\,\frac{1}{R^d}\,
\left\{\int_{t+Q_R}\,1\,d\mu(\xi) +\left|\int_{t+Q_R}\,e^{-2\pi i (x_1-x_2)\cdot \xi}\,d\mu(\xi)\right|\right\}.
$$
}
\end{example}

We need to prove a few technical lemmas before getting to the main result of this section,
\begin{lemma}\label{sum} Let $\psi \in \mathcal{S}(\mathbb{R}^d)$. Then, there exists $C>0$ such that
$$
\delta^d\,\sum_{k\in \mathbb{Z}^d} \,\sup_{\gamma \in Q_\delta }\,|\psi(\xi-k\delta-\gamma)|\le C,
\quad \xi \in\mathbb{R}^d,\,\,0<\delta\le 1.
$$
\end{lemma}
\begin{proof}
 Let us define
$$
g(\gamma)=\frac{1}{1+\gamma^2},\quad \gamma\in \mathbb{R}.
$$
and suppose that $0<\delta\le 1$.
If $\xi\in [-\delta/2, \delta/2]$ and $k\in \mathbb{Z}\setminus \{0\}$,
$$
\inf_{|\gamma|\le \delta/2}\,|\xi-\delta k-\gamma|
=\min\{|\xi-k\delta-\delta/2|,|\xi-k\delta+\delta/2|\} \ge \delta\,(|k|-1).
$$
Hence,
\begin{align*}
\delta \,\sum_{k\in \mathbb{Z}} \,\sup_{|\gamma|\le \delta/2  }\,g(\xi-k\delta-\gamma)
& \le \left(\delta+\sum_{k\in\mathbb{Z}\setminus \{0\}}\,
\frac{\delta}{1+\delta^2\,(|k|-1)^2}\right)\\
&=3\,\delta+ 2\,\sum_{n=1}^\infty\,\frac{\delta}{1+\delta^2\,n^2}\le
3+2\,\int_0^\infty\,\frac{1}{1+x^2}\,dx=c<\infty.
\end{align*}
Since the left-hand side of the previous expression is $\delta$-periodic, if follows that the inequality
holds for all $\xi\in \mathbb{R}$.
If $\psi \in \mathcal{S}(\mathbb{R}^d)$,  we have the estimate
$$
|\psi(\gamma)|\le C_1\,\prod_{i=1}^d\,g(\gamma_i),\quad \gamma=(\gamma_1,\dots,\gamma_d)\in \mathbb{R}^d.
$$
Therefore, for any $\xi\in \mathbb{R}^d$, we obtain
\begin{align*}
\delta^d\,\sum_{k\in \mathbb{Z}^d} \,\sup_{\gamma \in Q_\delta }\,
|\psi(\xi-\delta k-\gamma)|
& \le \sum_{k\in \mathbb{Z}^d} \,
C_1\, \prod_{i=1}^d\,\delta\, \sup_{|\gamma_i|\le \delta /2 }\,g(\xi_i-\delta k_i-\gamma_i)\\
&=C_1\, \prod_{i=1}^d\,\delta\, \sum_{k_i\in \mathbb{Z}}\,\sup_{|\gamma_i|\le \delta/2 }\,
g(\xi_i-\delta k_i-\gamma_i)\le C_1\,c^d=C<\infty.
\end{align*}
 \end{proof}
\begin{lemma}\label{lem2.2}
Let $\mu$ be a locally finite, positive  Borel measure on $\mathbb{R}^d$ and let $\delta>0$.
Suppose that $F_1,\dots,F_N \in \mathbb{R}^d$
are compactly supported. Then, for any $\epsilon>0$, we have the inequalities
$$
G(\delta,\epsilon)-I(\delta,\epsilon)\leq
\left(\int_{\mathbb{R}^d}\,\epsilon^d\, \bigg|\sum_{i=1}^{N}\,
\hat{F_i}(\epsilon\lambda)\,e^{-2\pi i x_i\cdot\lambda}\bigg|^2\,d\mu(\lambda)\right)^{1/2}
\leq G(\delta,\epsilon)+ I(\delta,\epsilon)
$$
if $I(\delta,\epsilon)<\infty$,
where
$$
G(\delta,\epsilon)=\left(\sum_{k\in\mathbb{Z}^d}
\int_{\delta k/\epsilon +Q_{\delta/\epsilon}}\,\epsilon^d\,\bigg|\sum_{i=1}^{N}\hat{F_i}(\delta k)
e^{-2\pi i x_i\cdot\lambda}\bigg|^2\,d\mu(\lambda)\right)^{1/2}
$$
and
$$ I(\delta,\epsilon)=
\left(\sum_{k\in \mathbb{Z}^d}\,\int_{\delta k/\epsilon+Q_{\delta/\epsilon}}\,\epsilon^d
\,\bigg|\sum_{i=1}^{N}\,\left(\hat{F_i}(\epsilon\lambda)-\hat{F_i}(\delta k)\right)\,
e^{-2\pi i x_i\cdot\lambda}\bigg|^2\,d\mu(\lambda)\right)^{1/2}.
$$
\end{lemma}

\begin{proof}
If $k\in  \mathbb{Z}^d$, define
$$
S_k(\gamma)=\epsilon^{d/2}\,\sum_{i=1}^{N}\hat{F_i}(\delta k)
e^{-2\pi i x_i\cdot\lambda}\quad\text{and}\quad
R_k(\gamma)=\epsilon^{d/2}\,\sum_{i=1}^{N}\,\left(\hat{F_i}(\epsilon\lambda)-\hat{F_i}(\delta k)\right)\,
e^{-2\pi i x_i\cdot\lambda}
$$
Applying Minkowski's inequality twice, we have
\begin{align*}
&\left(\int_{\mathbb{R}^d}\,\epsilon^d\, \bigg|\sum_{i=1}^{N}\,
\hat{F_i}(\epsilon\lambda)\,e^{-2\pi i x_i\cdot\lambda}\bigg|^2\,d\mu(\lambda)\right)^{1/2}
=\left(\sum_{k\in \mathbb{Z}^d}\,\int_{\delta k/\epsilon+Q_{\delta/\epsilon}}
\,|S_k(\gamma)+R_k(\gamma)|^2\,d\mu(\lambda)\right)^{1/2}\\
&\le \left(\sum_{k\in \mathbb{Z}^d}\,\left[\left(\int_{\delta k/\epsilon+Q_{\delta/\epsilon}}
\,|S_k(\gamma)|^2\,d\mu(\lambda)\right)^{1/2}+\left(\int_{\delta k/\epsilon+Q_{\delta/\epsilon}}
\,|R_k(\gamma)|^2\,d\mu(\lambda)\right)^{1/2}\right]^{2}\right)^{1/2}\\
&\le \left(\sum_{k\in \mathbb{Z}^d}\,\int_{\delta k/\epsilon+Q_{\delta/\epsilon}}
\,|S_k(\gamma)|^2\,d\mu(\lambda)\right)^{1/2}
+\left(\sum_{k\in \mathbb{Z}^d}\,\int_{\delta k/\epsilon+Q_{\delta/\epsilon}}
\,|R_k(\gamma)|^2\,d\mu(\lambda)\right)^{1/2}\\
&= G(\delta,\epsilon)+I(\delta,\epsilon).
\end{align*}

Similarly, reversing the role of $\delta k$ and $\lambda$ in the above computation, we obtain
the inequality
$$
G(\delta,\epsilon)
\leq \left(\int_{\mathbb{R}^d}\,\epsilon^d \left|\sum_{i=1}^{N}\,
\hat{F_i}(\epsilon\lambda)\,e^{-2\pi i x_i\cdot\lambda}\right|^2\,d\mu(\lambda)\right)^{1/2}
+ I(\delta,\epsilon).
$$
This completes the proof of this lemma.
\end{proof}

\begin{lemma}\label{lem2.3} Suppose that $\mu$ is a locally finite, positive  Borel measure on
 $\mathbb{R}^d$ and consider $N$ functions $F_1,\dots,F_N \in L^2(Q_1)$.
If $0<\delta\le 1$ and $\gamma\in\mathbb{R}^d$,   let
$\mathbf{v}_{\gamma} =  (\hat{F_1}(\gamma),\cdots,\hat{F_N}(\gamma))\in \mathbb{C}^N$.
Then, there exists a constant $C>0$ such that
$$
 |I(\delta,\epsilon)|^2
\leq C\,\delta\,
\int_{\mathbb{R}^d}\, \sup_{\zeta\in\mathbb{R}^d}\,
\frac{\check\mu_{\mathbf{v}_{\gamma}}(\zeta+Q_{\delta/\epsilon})}{(\delta/\epsilon)^d}\,d\gamma,
\quad \epsilon>0,
$$
where
$$
I(\delta,\epsilon)=
\left(\sum_{k\in \mathbb{Z}^d}\,\int_{\delta k/\epsilon+Q_{\delta/\epsilon}}\,\epsilon^d
\,\bigg|\sum_{i=1}^{N}\,\left(\hat{F_i}(\epsilon\lambda)-\hat{F_i}(\delta k)\right)\,
e^{-2\pi i x_i\cdot\lambda}\bigg|^2\,d\mu(\lambda)\right)^{1/2}.
$$
\end{lemma}
\begin{proof}
Let $\beta\in C^{\infty}_{0}(\mathbb{R}^d)$ with $\beta =1$ on a
neighborhood of $Q_1$ and let $\psi = \hat{\beta}$.
Then $\psi\in{\mathcal S}(\mathbb{R}^d)$ and $F\beta=F$ for any $F\in L^2(Q_1)$
which implies that $\hat{F}\ast\psi=\hat{F}$.  Using this last identity together with
 the Cauchy-Schwarz inequality, we have
 \begin{align*}
 &\left|\sum_{i=1}^{N}\,\left(\hat{F_i}(\epsilon\lambda)-\hat{F_i}(\delta k)\right)\,
e^{-2\pi i x_i\cdot\lambda}\right|^2 \\
&= \left|\int_{\mathbb{R}^d}\,\left(\psi(\epsilon\lambda-\gamma)-\psi(\delta k-\gamma)\right)
\,\bigg(\sum_{i=1}^{N}\,\hat{F_i}(\gamma)\,e^{-2\pi i x_i\cdot\lambda}\bigg)\,d\gamma\right|^2\\
 &\leq\bigg(\int_{\mathbb{R}^d}\,\left|\psi(\epsilon\lambda-\gamma)-\psi(\delta k-\gamma)\right|\,
\bigg|\sum_{i=1}^{N}\,\hat{F_i}(\gamma)\,e^{-2\pi i  x_i\cdot\lambda}\bigg|^2\,d\gamma
 \bigg) \times\\
&\qquad\left( \int_{\mathbb{R}^d}\,\left|\psi(\epsilon\lambda-\gamma)-\psi(\delta k-\gamma)\right|
\,d\gamma\right)\\
 &\leq 2\,\|\psi\|_{L^1}\,\int_{\mathbb{R}^d}\,
\left|\psi(\epsilon\lambda-\gamma)-\psi(\delta k-\gamma)\right|\,
\bigg|\sum_{i=1}^{N}\,\hat{F_i}(\gamma)\,e^{-2\pi i x_i\cdot\lambda}\bigg|^2\,d\gamma.\\
 \end{align*}
 Given $\epsilon>0$, we
consider the element $\check{\mu}^\epsilon$ of $\mathcal{MP}_N( \mathbb{R}^d)$
associated via formula (\ref{matrix-def}) to the measure $\mu^\epsilon$
 defined by
$$
\int_{\mathbb{R}^d}\,\phi(\xi)\,d\mu^{\epsilon}(\xi)=
\int_{\mathbb{R}^d}\,\epsilon^d\,\phi(\epsilon \xi)\,d\mu(\xi),\quad \phi\in C_c(\mathbb{R}^d).
$$
In particular, we have, for any $\mathbf{v}\in \mathbb{C}^N$ that
$$
\check\mu_{\mathbf{v}}^{\epsilon}(\xi+Q_{\delta})=\epsilon^d\,
\sum_{i,j=1}^{N}\,v_i\,\overline{v_j}\,\mu_{i,j}(\xi/\epsilon+Q_{\delta/\epsilon})
 = \delta^{d}\,\frac{\check\mu_{\mathbf{v}}(\xi/\epsilon+Q_{\delta/\epsilon})}{(\delta/\epsilon)^d},
\quad \xi\in \mathbb{R}^d,\,\,\epsilon,\delta>0,
$$
which yields the inequalities
\begin{equation}\label{2.2+}
\delta^d\,\inf_{\xi'\in\mathbb{R}^d}\,
\frac{\check\mu_{{\bf v}}(\xi'+Q_{\delta/\epsilon})}{(\delta/\epsilon)^d}
\leq\check\mu_{{\bf v}}^{\epsilon}(\xi+Q_{\delta})
\leq \delta^d\,\sup_{\xi'\in\mathbb{R}^d}\,
\frac{\check\mu_{{\bf v}}(\xi'+Q_{\delta/\epsilon})}{(\delta/\epsilon)^d}.
\end{equation}
 Hence, using Fubini's theorem and letting $C_0=2\,\|\psi\|_{L^1}$, we have
 \begin{align*}
 &|I(\delta,\epsilon)|^2\leq
C_0\,\sum_{k\in \mathbb{Z}^d}\,
\int_{\delta k/\epsilon+Q_{\delta/\epsilon}}\,
\int_{\mathbb{R}^d}\,\epsilon^d\,\left|\psi(\epsilon\lambda-\gamma)-\psi(\delta k-\gamma)\right|
\,\bigg|\sum_{i=1}^{N}\,\hat{F_i}(\gamma)\,
e^{-2\pi i x_i\cdot\lambda}\bigg|^2\,d\gamma\, d\mu(\lambda)\\
 %&=C_0\,\sum_{k\in\mathbb{Z}^d}\,\int_{\mathbb{R}^d}\,
%\int_{\mathbb{R}^d}\,\epsilon^d\,\chi_{\delta k+Q_{\delta}}(\epsilon\lambda)\,
%\left|\psi(\epsilon\lambda-\gamma)-\psi(\delta k-\gamma)\right|
%\,\sum_{i,j=1}^{N}\,\hat{F_i}(\gamma)\,\overline{\hat{F_j}(\gamma)}\,
%e^{-2\pi i (x_i-x_j)\cdot\lambda}\, d\mu(\lambda)\,d\gamma\\
 &=C_0\,\sum_{k\in \mathbb{Z}^d}\,\int_{\mathbb{R}^d}\,
\sum_{i,j=1}^{N}\,\int_{\delta k+Q_{\delta}}\,\left|\psi(\lambda-\gamma)-\psi(\delta k-\gamma)\right|
\,\hat{F_i}(\gamma)\,\overline{\hat{F_j}(\gamma)}\,d\mu^{\epsilon}_{i,j}(\lambda)\,d\gamma\\
 &=C_0\,\int_{\mathbb{R}^d}\,\sum_{k\in \mathbb{Z}^d}\,
\int_{\delta k+Q_{\delta}}\,\left|\psi(\lambda-\gamma)-\psi(\delta k-\gamma)\right|
\,d\check\mu_{\mathbf{v}_{\gamma}}^{\epsilon}(\lambda)\,d\gamma.
\end{align*}
Using the mean-value theorem, we have the estimate
 $$
 |\psi(\lambda-\gamma)-\psi(\delta k-\gamma)|\leq
\delta\sqrt{d}\sum_{i=1}^{d}\sup_{\xi'\in\delta k+Q_{\delta}}
\left|\frac{\partial\psi}{\partial\xi_i}(\xi'-\gamma)\right|,\quad \lambda\in\delta k+Q_{\delta}.
 $$
 Hence, by (\ref{2.2+}),
 $$
 |I(\delta,\epsilon)|^2\leq
\int_{\mathbb{R}^d}\, C_0\,\delta\,\sqrt{d}\,
\left(\sup_{\zeta\in\mathbb{R}^d}\,
\frac{\mu_{{\bf v}_{\gamma}}(\zeta+Q_{\delta/\epsilon})}{(\delta/\epsilon)^d}\right)\,
\sum_{i=1}^{d}\,\sum_{k\in \mathbb{Z}^d}\,
\sup_{\xi'\in\delta k+Q_{\delta}}\delta^d\,
\left|\frac{\partial\psi}{\partial\xi_i}(\xi'-\gamma)\right|\,d\gamma
 $$
 Applying Lemma \ref{sum} to each function
$\frac{\partial\psi}{\partial\xi_i}$ in $\mathcal{S}(\mathbb{R}^d)$,
we can thus  find a constant $C>0$ such that
 $$
 |I(\delta,\epsilon)|^2\leq C\,\delta\, \int_{\mathbb{R}^d}\,
\sup_{\zeta\in\mathbb{R}^d}\,
\frac{\mu_{{\bf v}_{\gamma}}(\zeta+Q_{\delta/\epsilon})}{(\delta/\epsilon)^d}\,d\gamma,
\quad \epsilon>0,
 $$
as claimed.
 \end{proof}

We now state the main result of this section.

\begin{theorem}\label{th1.1}
Let $x_1,\cdots, x_N \in\mathbb{R}^d$ be distinct and let $\mu$ be a locally finite, positive Borel measure
on $\mathbb{R}^d$. Define the associate positive matrix-valued measure $\check \mu$
using formula (\ref{matrix-def}).
Then the following are equivalent.
\begin{enumerate}[(a)]
\item There exist constants $A,B>0$ and $\epsilon>0$ such that the sets $x_j+Q_{\epsilon}$, $j=1,\dots,N$,
are disjoint and such that  the frame inequalities
\begin{equation}\label{eq2.2}
A\,\|f\|_2^2\leq \int_{\mathbb{R}^d}|\hat{f}(\lambda)|^2\,d\mu(\lambda)
\leq B\, \|f\|_2^2, \ f\in L^2\left(\Omega\right),
\end{equation}
are satisfied for $\Omega=\bigcup_{j=1}^{N}(x_j+Q_{\epsilon})$.
\item We have $0<\mathcal{D}^{-}_N(\check{\mu})\leq \mathcal{D}^{+}_N(\check{\mu}) <\infty$.
\end{enumerate}
Moreover, if (a) holds, we have the inequalities
$A\leq \mathcal{D}^{-}_N(\check{\mu})\leq \mathcal{D}^{+}_N(\check{\mu})\leq B$.
Conversely, if (b) holds, then, for any
$\rho>0$ such that $\mathcal{D}^{-}_N(\check{\mu})-\rho>0$, there exists
$\epsilon>0$ such that (a) holds with $A=\mathcal{D}^{-}_N(\check{\mu})-\rho$ and
$B=\mathcal{D}^{-}_N(\check{\mu})+\rho$.
\end{theorem}

\begin{proof}
If $\xi \in\mathbb{R}^d$, define the modulation operator $M_\xi$ acting on $L^2(\mathbb{R}^d)$, by
$$
M_{\xi} f(x) = e^{2\pi i\xi\cdot x}\,f(x),\quad f\in L^2(\mathbb{R}^d).
$$
If (a) holds,
let $\mathbf{v}=(v_1,\dots,v_N)\in \mathbb{C}^N$ with $\|\mathbf{v} \|_2=1$, let
$\xi \in \mathbb{R}^d$ and, if $f\in  L^2(Q_{\epsilon})$,
consider the function
$$
h:=\sum_{i=1}^{N}\,v_i\,\delta_{x_i}\ast (M_{\xi} \overline{f})\in L^2(\Omega).
$$
  Then,
$$
\int_{\Omega}\,|h(x)|^2\,dx =
\left(\sum_{i=1}^{N}\,|v_i|^2\right)\,\left(\int_{Q_{\epsilon}}\,|f(x)|^2\,dx\right) =
\int_{Q_{\epsilon}}\,|f(x)|^2\,dx.
$$
On the other hand, we have also
\begin{align*}
\int_{\mathbb{R}^d}\,|\hat{h}(\lambda)|^2\,d\mu(\lambda)
=& \int_{\mathbb{R}^d}\,\left|\sum_{i=1}^{N}\,v_i\,\hat{f}(\xi-\lambda)\,
e^{-2\pi i  x_i\cdot\lambda}\right|^2\,d\mu(\lambda)
= \sum_{i,j=1}^{N}\,\int_{\mathbb{R}^d}\,v_i\,\overline{v_j}\,
|\hat{f}(\xi-\lambda)|^2\,d\mu_{i,j}(\lambda)\\
=&\int_{\mathbb{R}^d}\,|\hat{f}(\xi-\lambda)|^2\,d\check\mu_{{\bf v}}(\lambda)
= \left(\check\mu_{{\bf v}}\ast|\hat{f}|^2\right)(\xi).\\
\end{align*}
Using the frame inequalities (\ref{eq2.2}), we have thus
$$
A\,\int_{Q_{\epsilon}}\,|f(x)|^2\,dx
\leq \check\mu_{{\bf v}}\ast|\hat{f}|^2 \leq B\,\int_{Q_{\epsilon}}|f(x)|^2\,dx.
$$
As $\check\mu_{{\bf v}}$ is a positive measure, Theorem \ref{Ga-th1} shows that
$0<A\leq \mathcal{D}^{-}(\check\mu_{\mathbf{v}})\leq
\mathcal{D}^{+}(\check\mu_{\mathbf{v}})\leq B<\infty$ for any
vector $\mathbf{v}$ of
norm 1. Taking infimum and supremum, respectively,  over the unit ball of  $\mathbb{C}^N$,
we obtain (b).

We now prove that (b) implies (a).
Let $\Omega =
\bigcup_{i=1}^{N}\left(x_i+Q_{\epsilon}\right)$
with $\epsilon>0$ chosen small enough so the sets in the previous union
are pairwise disjoint.
A function $f\in L^2(\Omega)$  can then be uniquely written as
$$
f = \sum_{i=1}^{N}\,\delta_{x_i}\ast f_i,\quad \ f_i\in L^2(Q_{\epsilon}).
$$
For each of the functions $f_i$ above, let $F_i\in L^2(Q_1)$ be defined by
$f_i(x)= {\epsilon}^{-d/2}\,F_i(x/\epsilon)$. Clearly $\|F_i\|_2 = \|f_i\|_2$ and
 \begin{equation}\label{eq2.3}
 \int_{\Omega}\,|f(x)|^2\,dx=
\sum_{i=1}^{N}\, \int_{\mathbb{R}^d}\,|\hat F_i(\lambda)|^2\,d\lambda\quad
\text{with}\quad
\hat{f}(\xi)= \epsilon^{d/2}\,\sum_{i=1}^{N}\,
\hat{F_i}(\epsilon \xi)\,e^{-2\pi i x_i\cdot\xi}.
 \end{equation}
 Hence, proving (\ref{eq2.2}) is equivalent to showing that
for any $F_1,\dots,F_N\in L^2(Q_1)$ and for $\epsilon>0$ sufficiently small,
we have the inequalities
\begin{equation}\label{eq2.5}
A\,\sum_{i=1}^{N}\,\int_{\mathbb{R}^d}\,|\hat{F_i}(\lambda)|^2\,d\lambda
\leq \int_{\mathbb{R}^d}\,\epsilon^d\, \left|\sum_{i=1}^{N}\,
\hat{F_i}(\epsilon\lambda)\,e^{-2\pi i  x_i\cdot\lambda}\right|^2\,d\mu(\lambda)
\leq B\, \sum_{i=1}^{N}\int_{\mathbb{R}^d}\,|\hat{F_i}(\lambda)|^2\,d\lambda.
\end{equation}

 Given any number $\rho>0$ with  $\mathcal{D}^{-}_N(\check{\mu})-\rho>0$,
we choose $\rho'$ with $0<\rho'\le 1$ and small enough so that
$$
(\sqrt{ \mathcal{D}^{+}_N(\check{\mu})+\rho'}+\sqrt{\rho'})^2\leq
(\mathcal{D}^{+}_N(\check{\mu})+\rho)
\quad\text{and}\quad
 (\sqrt{\mathcal{D}^{-}_N(\check{\mu})-\rho'}-\sqrt{\rho'})^2\geq
\mathcal{D}^{-}_N(\check{\mu})-\rho.
$$
With this particular  chosen $\rho'>0$,
we can find a number $\delta$ with $0<\delta\le 1$ and  small enough so that
$$
4\,C\,\delta\, ( \mathcal{D}^{+}_N(\check{\mu}))+ \mathcal{D}^{-}_N(\check{\mu})+1)<\rho',
$$
where $C$ is the constant obtained in Lemma \ref{lem2.3}.
With that value of $\delta$ fixed, we use the assumption (b)
and the definition of $\mathcal{D}^{\pm}_N(\check{\mu})$ to obtain the existence of a number
 $\epsilon>0$ small enough so that
$$
\sup_{\zeta\in\mathbb{R}^d}\frac{\check\mu_{{\bf v}}(\zeta+Q_{\delta/\epsilon})}{(\delta/\epsilon)^d}
\leq (\mathcal{D}^{+}_N(\check{\mu}))+\rho')\|{\bf v}\|^2.
$$
and
$$
\inf_{\zeta\in\mathbb{R}^d}\frac{\check\mu_{{\bf v}}(\zeta+Q_{\delta/\epsilon})}{(\delta/\epsilon)^d}
\geq (\mathcal{D}^{-}_N(\check{\mu})-\rho')\|{\bf v}\|^2
$$
for any $\|{\bf v}\|=1$. If $\gamma\in\mathbb{R}^d$,   let us denote by
$\mathbf{v}_{\gamma}$ the vector $(\hat{F_1}(\gamma),\cdots,\hat{F_N}(\gamma))\in \mathbb{C}^N$,
Applying Lemma \ref{lem2.3} to the functions $F_1,\dots,F_N\in L^2(Q_1)$
defined above and letting
$$
I(\delta,\epsilon)=
\left(\sum_{k\in \mathbb{Z}^d}\,\int_{\delta k/\epsilon+Q_{\delta/\epsilon}}\,\epsilon^d
\,\bigg|\sum_{i=1}^{N}\,\left(\hat{F_i}(\epsilon\lambda)-\hat{F_i}(\delta k)\right)\,
e^{-2\pi i x_i\cdot\lambda}\bigg|^2\,d\mu(\lambda)\right)^{1/2},
$$
we obtain that
\begin{align*}
 |I(\delta,\epsilon)|^2&\leq C\,\delta \,
\int_{\mathbb{R}^d}\, \sup_{\zeta\in\mathbb{R}^d}
\frac{\check\mu_{\mathbf{v}_{\gamma}}(\zeta+Q_{\delta/\epsilon})}{(\delta/\epsilon)^d}\,d\gamma
\leq C\,\delta\,\left(\mathcal{D}^{+}_N(\check{\mu})+\rho'\right)\,
\int_{\mathbb{R}^d}\, \|\mathbf{v}_{\gamma}\|^2\,d\gamma\\
\leq& \rho'\,\int_{\mathbb{R}^d}\, \sum_{i=1}^{N}\,|\hat{F_i}(\gamma)|^2\,d\gamma = \rho'\|f\|_2^2.
\end{align*}

On the other hand, letting
$$
|G(\delta,\epsilon)|^2=\sum_{k\in \mathbb{Z}^d}\,
\int_{\delta k/\epsilon+
Q_{\delta/\epsilon}}\,\epsilon^d\,\left|\sum_{i=1}^{N}\,\hat{F_i}(\delta k)
e^{-2\pi i x_i\cdot\lambda}\right|^2\,d\mu(\lambda)
$$
we have
\begin{align*}
|G(\delta,\epsilon)|^2
=&\sum_{k\in \mathbb{Z}^d}\,\sum_{i,j=1}^{N}\,\hat{F_i}(\delta k)\,
\overline{\hat{F_j}(\delta k)}\,\epsilon^d\,\mu_{i,j}(\delta k/\epsilon+Q_{\delta/\epsilon})
= \sum_{k\in \mathbb{Z}^d}\,\epsilon^d\,\check\mu_{\mathbf{v}_{\delta k}}(\delta/\epsilon
+Q_{\delta/\epsilon})\\
\leq& \delta^d\,\sum_{k\in\mathbb{Z}^d}\,\sup_{\zeta\in\mathbb{R}^d}\,
\frac{\check\mu_{\mathbf{v}_{\delta k}}(\zeta+Q_{\delta/\epsilon})}{(\delta/\epsilon)^d}
\leq\delta^d\,\sum_{k\in \mathbb{Z}^d}\,(\mathcal{D}^{+}_N(\check{\mu})+\rho')\,
\|\mathbf{v}_{\delta k}\|_2^2\\
=&(\mathcal{D}^{+}_N(\check{\mu})+\rho')\,\sum_{i=1}^{N}\,
\sum_{k\in\mathbb{Z}^d}\,\delta^d\,|\hat{F_i}(\delta k)|^2.
\end{align*}
Since each function $F_i$ is supported in $Q_1$ and $\delta\leq 1$,
the Shannon sampling theorem shows that
$$
\sum_{k\in \mathbb{Z}^d}\,\delta^d\,|\hat{F_i}(\delta k)|^2 = \|F_i\|_2^2,\quad i=1,\dots N.
$$
Hence, we obtain the inequality
$$
|G(\delta,\epsilon)|^2\leq (\mathcal{D}^{+}_N(\check{\mu}))+\rho')\,\|f\|_2^2.
$$
A similar computation shows also that
$$
|G(\delta,\epsilon)|^2\geq (\mathcal{D}^{+}_N(\check{\mu}))-\rho')\,\|f\|_2^2.
$$
Using the estimates for $I(\delta,\epsilon)$ and $G(\delta,\epsilon)$
just obtained, we deduce from Lemma \ref{lem2.2}, that

\begin{align*}
&\int_{\mathbb{R}^d}\,\epsilon^d\, \bigg|\sum_{i=1}^{N}\,
\hat{F_i}(\epsilon\lambda)\,e^{-2\pi i x_i\cdot\lambda}\bigg|^2\,d\mu(\lambda)
\leq \left(G(\delta,\epsilon)+I(\delta,\epsilon)\right)^2\\
&\le
\left(\sqrt{\mathcal{D}^{+}_N(\check{\mu})+\rho'}+\sqrt{\rho'}\right)^2\,\|f\|_2^2
\leq (\mathcal{D}^{+}_N(\check{\mu})+\rho)\,\|f\|_2^2
\end{align*}
and
\begin{align*}
&\int_{\mathbb{R}^d}\,\epsilon^d\, \bigg|\sum_{i=1}^{N}\,
\hat{F_i}(\epsilon\lambda)\,e^{-2\pi i x_i\cdot\lambda}\bigg|^2\,d\mu(\lambda)
\geq \left(G(\delta,\epsilon)-I(\delta,\epsilon)\right)^2\\
&\ge
\left(\sqrt{\mathcal{D}^{-}_N(\check{\mu})-\rho'}-\sqrt{\rho'}\right)^2\,\|f\|_2^2
\geq (\mathcal{D}^{-}_N(\check{\mu})-\rho)\,\|f\|_2^2.
\end{align*}
This completes the  proof.

\end{proof}
Of course, ignoring the lower-bound estimates in the proof just given,
we can also prove the Bessel version of the previous theorem.
\begin{theorem}\label{th1.2}
Let $x_1,\cdots, x_N \in\mathbb{R}^d$ be distinct and let $\mu$ be a locally finite, positive Borel measure
on $\mathbb{R}^d$. Define the associate positive matrix-valued measure $\check \mu$
using formula (\ref{matrix-def}).
Then the following are equivalent.
\begin{enumerate}[(a)]
\item There exist constants $A>0$ and $\epsilon>0$ such that the sets $x_j+Q_{\epsilon}$, $j=1,\dots,N$,
are disjoint and such that  the Bessel inequality
\begin{equation}\label{bessel2}
\int_{\mathbb{R}^d}|\hat{f}(\lambda)|^2\,d\mu(\lambda)
\leq B\, \|f\|_2^2, \ f\in L^2\left(\Omega\right)
\end{equation}
holds for $\Omega=\bigcup_{j=1}^{N}(x_j+Q_{\epsilon})$.
\item We have $ \mathcal{D}^{+}_N(\check{\mu}) <\infty$.
\end{enumerate}
Moreover, if (a) holds, we have the inequality
$\mathcal{D}^{+}_N(\check{\mu})\leq B$.
Conversely, if (b) holds, then, for any
$\rho>0$,  there exists $\epsilon>0$ such that (a) holds with
$B=\mathcal{D}^{-}_N(\check{\mu})+\rho$.
\end{theorem}

\medskip

The following gives an interpretation of the
lower and upper density of the positive matrix-valued measure $\check \mu$
defined (\ref{matrix-def}) as limiting lower  and upper
frame bounds, respectively.

\begin{corollary}\label{cor2.2}
Under the assumptions of Theorem \ref{th1.1}, suppose that $\mathcal{D}^{+}_N(\check{\mu})<\infty$.
Define $A_{\epsilon}$ and $B_{\epsilon}$ to be the optimal bounds for
the inequalities
\begin{equation}
A_{\epsilon}\,\|f\|_2^2\leq \int_{\mathbb{R}^d}|\hat{f}(\lambda)|^2\,d\mu(\lambda)
\leq B_{\epsilon}\, \|f\|_2^2, \ f\in L^2\left(\Omega\right),
\end{equation}
where $\Omega=\bigcup_{j=1}^{N}(x_j+Q_{\epsilon})$.
Then $\lim_{\epsilon\rightarrow 0} A_{\epsilon}=
\mathcal{D}^{-}_N(\check{\mu})$
and
$\lim_{\epsilon\rightarrow 0} B_{\epsilon}=\mathcal{D}^{+}_N(\check{\mu})$.
\end{corollary}

\begin{proof} If $\mathcal{D}^{-}_N(\check{\mu})>0$, the inequalities obtained in
Theorem \ref{th1.1}, show,  for any $\rho>0$, that
$$
\mathcal{D}^{-}_N(\check{\mu})-\rho \le A_{\epsilon}\le \mathcal{D}^{-}_N(\check{\mu})
\quad \text{and}\quad \mathcal{D}^{+}_N(\check{\mu}) \le B_{\epsilon}\le \mathcal{D}^{+}_N(\check{\mu})+\rho,
$$
if $\epsilon>0$ is small enough, proving our statement in that case.
If $\mathcal{D}^{-}_N(\check{\mu})=0$ and $\rho>0$, we have $A_{\epsilon}=0$
and the inequalities $\mathcal{D}^{+}_N(\check{\mu}) \le B_{\epsilon}\le \mathcal{D}^{+}_N(\check{\mu})+\rho$
 if $\epsilon>0$ is small enough, from which our claim follows immediately.
\end{proof}

\begin{remark} We note that using (\ref{eq3.2}), given a positive Borel measure $\mu$ on $\mathbb{R}^d$, the quantity
$\mathcal{D}^{+}_N(\check{\mu})$, where  $\check \mu$ is
defined in (\ref{matrix-def}) can also be defined as the smallest constant $B\ge 0$ such hat
$$
\limsup_{h\to \infty}\,\sup_{t\in \mathbb{R}^d}\,
\frac{1}{h^d}\,\int_{t+Q_h}\,
\big|\sum_{i=1}^m\,a_i\,
e^{2\pi i x_i\cdot \xi}\big|^2\,d\mu(\xi)\le B\,\sum_{i=1}^m\,|a_i|^2,
$$
for any  $a_1,\dots,a_m\in \mathbb{C}$.
Similarly, if $\mathcal{D}^{+}_N(\check{\mu})<\infty$, the quantity $\mathcal{D}^{-}_N(\check{\mu})$
can also be defined as the largest constant $A\ge 0$ such hat
$$
\liminf_{h\to \infty}\,\inf_{t\in \mathbb{R}^d}\,
\frac{1}{h^d}\,\int_{t+Q_h}\,
\big|\sum_{i=1}^m\,a_i\,
e^{2\pi i x_i\cdot \xi}\big|^2\,d\mu(\xi)\ge A\,\sum_{i=1}^m\,|a_i|^2,
$$
for any  $a_1,\dots,a_m\in \mathbb{C}$. We will use these alternate definitions in the next section.
\end{remark}
\medskip

\section{Uniform limiting frame bounds for subgroups of $\mathbb{R}$.}

This section will be devoted to the characterization of uniform ${\mathcal F}$-measures and ${\mathcal B}$-measures
 over subgroups of ${\mathbb R}$. These are the elements of the sets ${\mathcal F}(G,A,B)$ and ${\mathcal B}(G,B)$, respectively,
 defined in Definition \ref{def3} of the introduction.  Recall that our goal will be to characterize the measures having the property of being a common $\mathcal{F}$-measure (resp.~$\mathcal{B}$-measure) for $L^2(\Omega)$, where
$\Omega$ is any set of the form $\Omega=\bigcup_{j=1}^{N}(x_j+Q_{\epsilon})$,
for $\epsilon>0$ small enough and dependent on the points $x_i$,
where the points $x_j$ belong to a given subgroup $G$ of $\mathbb{R}$,
but with the limiting frame bounds (resp. Bessel bounds), as defined in the end of the
last section, being independent of the points $x_i$, $i=1,\dots, N$.
%This motivates the following definitions.
%\begin{definition} Let $\mu$ positive Borel measure on $\mathbb{R}$ and let
%$G$ be a subgroup of $\mathbb{R}$. If $A,B>0$, we say that $\mu$ is a uniform $\mathcal{F}$-measure
%for $G$ with limiting lower bound larger than or equal to $A$ and limiting upper frame bound
%less than or equal to  $B$, if  given any $x_1,\dots ,x_M\in G$ and any $\delta>0$,
%there exists $\epsilon>0$ such that
% \begin{equation}\label{eq3.1}
%(A-\delta)\,\|f\|_2^2\leq \int_{\mathbb{R}^d}|\hat{f}(\lambda)|^2\,d\mu(\lambda)
%\leq (B+\delta)\, \|f\|_2^2, \ f\in L^2\left(\Omega\right),
%\end{equation}
%for $\Omega=\bigcup_{j=1}^{N}(x_j+Q_{\epsilon})$.
%We denote the collection of such measures by $\mathcal{F}(G,A,B)$.
%The notion of uniform $\mathcal{B}$-measure for $G$ with limiting upper frame bound
%less than or equal to  $B$ is defined in a similar way and the collection of such
%measures is denoted by $\mathcal{B}(G,B)$. Finally, the measures in the
%collection  $\mathcal{F}(G,A,A)$ are called  uniform $\mathcal{F}$-measures with limiting tight
%frame bound $A$
%for $G$.
%\end{definition}

In the problems stated, it is clear that a measure will satisfy one of the
properties mentioned with respect to a given group $G$ if and only if
it will do so for any one of its finitely generated subgroup, as we need only to check
 those properties on each finite subset of $G$. Therefore, our main focus will be to deal
with finitely generated subgroup of $\mathbb{R}$, i.e.~those of the form
$$
G=\left\{\sum_{i=1}^s\,m_i\,a_i,\,\,m_i\in \mathbb{Z}\right\}.
$$
where $a_1,\dots,a_s$ can be assumed to be linearly independent over $\mathbb{Q}$.

The notion of weak-$*$ convergence of measures defined below
and the property of  weak-$*$ compactness
  will play an important role
in the proofs of our main results. Let $C_c({\mathbb R}^s)$ denote the space of complex-valued continuous functions
with compact support defined on ${\mathbb R}^s$.

\begin{definition} Let $\sigma_i$, $i\ge 1$, and  $\sigma$ be locally finite, positive Borel measures
on $\mathbb{R}^s$. We say that $\sigma_i$ converges to  $\sigma$ in the weak-$*$ topology
as $i\to \infty$
if for any $\varphi\in C_c(\mathbb{R}^s)$, we have
$$
\lim_{i\to \infty}\,\int_{\mathbb{R}^s}\,\varphi(\xi)\,d\sigma_i(\xi)
=\int_{\mathbb{R}^s}\,\varphi(\xi)\,d\sigma(\xi).
$$
\end{definition}

We have the following criterion  for weak-$*$ compactness:
if $\{\sigma_i\}_{i\ge 1}$ is a sequence of
 locally finite, positive Borel measures
on $\mathbb{R}^s$ which is locally uniformly bounded, i.e.~for any compact $K\subset\mathbb{R}^s$,
there exists a constant $C(K)$ such that
$$
\sup_{i\ge 1}\,\sigma_i(K)\le C(K),
$$
then the sequence $\{\sigma_i\}_{i\ge 1}$ admits a subsequence which is convergent
in the weak-$*$ topology.
We need some preliminary lemmas.
\begin{lemma}\label{tau} Let $\{\tau_j\}_{j\ge 1}$ be a sequence of positive Borel measures on $\mathbb{R}^s$
such that $\tau_j\to \tau$ in the weak-$*$ topology.
If $\tau$ is absolutely continuous with respect to the Lebesgue measure with Radon-Nikodym
derivative $G\in L^\infty(\mathbb{R}^s)$, then we have
$$
\lim_{j\to \infty}\,\int_{F}\,H(\xi)\,d\tau_j(\xi)=\int_{F}\,H(\xi)\,d\tau(\xi)
$$
for any continuous function $H\ge 0$ on $\mathbb{R}^s$
where $F=I_1\times\dots\times I_s$ and $I_s\subset \mathbb{R}$, $j=1,\dots, s$, are bounded intervals.
\end{lemma}
\begin{proof} Let $B=\|G\|_{\infty}$. Given $\epsilon >0$, choose real-valued
functions $\phi_1, \phi_2\in C_c(\mathbb{R}^s)$ such that
$\phi_1(\xi)\le H(\xi) \,\chi_F(\xi)\le \phi_2(\xi)$, for $\xi\in \mathbb{R}$,
and with
$$
\int_{\mathbb{R}^s}\,\phi_2(\xi)-\phi_1(\xi)\,d\xi\le\epsilon.
$$
Then,
\begin{align*}
&\limsup_{j\to \infty}\,\int_{\mathbb{R}^s}\,\phi_2(\xi)-H(\xi)\,\chi_F(\xi)\,d\tau_j(\xi)
\le
\lim_{j\to \infty}\,\int_{\mathbb{R}^s}\,\phi_2(\xi)-\phi_1(\xi)\,d\tau_j(\xi)\\
&=\int_{\mathbb{R}^s}\,G(\xi)\,\left(\phi_2(\xi)-\phi_1(\xi)\right)\,d\xi
\le B\,\int_{\mathbb{R}^s}\,\left(\phi_2(\xi)-\phi_1(\xi)\right)\,d\xi\le B\,\epsilon
\end{align*}
and thus
$$
\liminf_{j\to \infty}\,\int_{F}\,H(\xi)\,d\tau_j(\xi)\ge
\int_{\mathbb{R}^s}\,G(\xi)\,\phi_2(\xi)\,d\xi-B\,\epsilon\ge
\int_{\mathbb{R}^s}\,G(\xi)\,H(\xi)\,\chi_F(\xi)\,d\xi-B\,\epsilon.
$$
Similarly,
\begin{align*}
\limsup_{j\to \infty}\,\int_{\mathbb{R}^s}\,H(\xi)\,\chi_F(\xi)-\phi_1(\xi)\,d\tau_j(\xi)&\le
\lim_{j\to \infty}\,\int_{\mathbb{R}^s}\,\phi_2(\xi)-\phi_1(\xi)\,d\tau_j(\xi)\le B\,\epsilon
\end{align*}
which shows  that
$$
\limsup_{j\to \infty}\,\int_{F}\,H(\xi)\,d\tau_j(\xi)\le
\int_{\mathbb{R}^s}\,G(\xi)\,\phi_1(\xi)\,d\xi+B\,\epsilon
\le \int_{\mathbb{R}^s}\,G(\xi)\,H(\xi)\,\chi_F(\xi)\,d\xi+B\,\epsilon.
$$
We have thus
\begin{align*}
\int_{F}\,H(\xi)\,G(\xi)\,d\xi- B\,\epsilon&\le
  \liminf_{j\to \infty}\,\int_{F}\,H(\xi)\,d\tau_j(\xi)
\le   \limsup_{j\to \infty}\,\int_{F}\,H(\xi)\,d\tau_j(\xi)\\
&\le \int_{F}\,H(\xi)\,G(\xi)\,d\xi+B\,\epsilon,
\end{align*}
and the result follows since $\epsilon>0$ is arbitrary and $d\tau=G(\xi)\,d\xi$.
\end{proof}

 If $\rho$ is a signed or complex measure on $\mathbb{R}$, we will
denote by $|\rho|$ its total variation.
\begin{lemma}\label{per} Let $a_1,\dots, a_s$ be $s$ positive real numbers and let
$\mu$ be a positive translation-bounded Borel measure on $\mathbb{R}$.
Consider the positive Borel measure $\nu_\mu$ on $\mathbb{R}^s$
defined by
\begin{equation}\label{nudef}
\int_{\mathbb{R}^s}\,\varphi(\xi_1,\dots,\xi_s)\,d\nu_\mu(\xi_1,\dots,\xi_s)=
\int_{\mathbb{R}}\,\varphi(\lambda,\dots,\lambda)\,d\mu(\lambda),\quad \varphi\in C_c(\mathbb{R}^s),
\end{equation}
and define
\begin{equation}\label{sigmadef}
\sigma_{c,R}:=\delta_{c}\ast \frac{1}{\,R}\,\sum_{0\le k_1\le a_1 R-1}\,
\dots\sum_{0\le k_s\le a_s R-1}\,
\delta_{(-k_1/a_1,\dots,-k_1/a_s)}\ast \nu_\mu
\end{equation}
where $c\in \mathbb{R}^s$ and $R>0$. Let  $\mathcal{L}$ denote the
 lattice $\prod_{k=1}^s\,a_k^{-1}\,\mathbb{Z}$. Then, the following properties hold.
\begin{enumerate}[(a)]
\item For any compact $K\subset\mathbb{R}^s$, the set $\{\sigma_{c,R}(K),\,\,c\in  \mathbb{R}^s,\,\,
R\ge 1\}$ is bounded.
\item
Let $\sigma_j:=\sigma_{c_j,R_j}$ where $R_j\to \infty$ and $c_j\in \mathbb{R}^s$.
If the sequence $\{\sigma_j\}_{j\ge 1}$ converges to the measure $\sigma$
in the weak-$*$ topology, then $\sigma$
is $\mathcal{L}$-periodic.
\end{enumerate}
\end{lemma}
\begin{proof} If $r>0$ and $c=(c_1,\dots,c_s)$,  we have
\begin{align*}
&\sigma_{c,R}([-r,r]^s)=\int_{\mathbb{R}^s}\,\chi_{[-r,r]^s}(\xi)\,d\sigma_{c,R}(\xi)\\
&=\frac{1}{\,R}\,\sum_{0\le k_1\le a_1 R-1}
\dots\sum_{0\le k_s\le a_s R-1}\,\int_{\mathbb{R}^s}\,
\chi_{[-r,r]^s}(\xi+c+(-k_1/a_1,-k_2/a_2,\dots,-k_1/a_s))\,d\nu_\mu(\xi)\\
&=\frac{1}{\,R}\,\sum_{0\le k_1\le a_1 R-1}
\dots\sum_{0\le k_s\le a_s R-1}\,\int_{\mathbb{R}}\,
\prod_{m=1}^s\,\chi_{[-r,r]}(\lambda-k_m/a_m+c_m)
\,d\mu(\lambda)
\end{align*}
Note that if for some integer $k_1$ and some $m\ge 2$ and some integer $k$, we have
$$
\left([-r,r]+k_1/a_1-c_1\right)\cap \left([-r,r]+k/a_m-c_m\right)\ne \emptyset,
$$
then $|k/a_m-b_m|\le  2\,r$, where $b_m=c_m-c_1+k_1/a_1$. Hence,
$|k-a_m b_m|\le 2\,r\,a_m$ and the number of integers $k$ satisfying this inequality
can be at most $4\,r\,a_m+1$.
It follows that, if $C=\sup_{t\in \mathbb{R}}\mu([-r,r]+t)<\infty$, we have
$$
\sigma_{c,R}([-r,r]^s)\le \frac{C}{\,R}\,\sum_{0\le k_1\le a_1 R-1}\, \prod_{m=2}^s (4\,r\,a_m+1)
\le C\,a_1\,\prod_{m=2}^s (4\,r\,a_m+1),
$$
which proves (a).
To prove (b), it is enough to show that if $\nu_{l,j}:=(\delta_l-\delta_0)\ast \sigma_j$ and,
for any $l\in \mathcal{L}$, $|\nu_{l,j}|(K)\to 0$ as
$j\to \infty$ for any compact set $K\subset  \mathbb{R}^s$.
Indeed, if it is the case, then $\nu_{l,j}\to 0$ in the weak-$*$ topology
and, in particular, if $\sigma_j\to \sigma$
in the weak-$*$ topology, then
$$
(\delta_l-\delta_0)\ast \sigma=\lim_{j\to \infty}\,\nu_{l,j}=0,
$$
showing that $\sigma$ is $\mathcal{L}$-periodic.
We consider first the case $l=a_i^{-1}\,e_i$,
where $e_i$,   $i=1,\dots,s$, is the standard basis in $\mathbb{R}^s$.
We only deal with the case $i=1$ since the other cases are
similar.
We have then
\begin{align*}
&\left(\delta_{e_1/a_1}-\delta_0\right)*\sigma_j\\
&=\delta_{c_j}\ast (\delta_{e_1\,N_1/a_1}-\delta_0)\ast \frac{1}{\,R_j}\,\sum_{0\le k_2\le a_1 R_j-1}\,
\dots \sum_{0\le k_s\le a_s R_j-1}\,
\delta_{(0,-k_2/a_2,\dots,-k_1/a_s)}\ast \nu_\mu,
\end{align*}
where $N_1=\lfloor a_1 R_1\rfloor$ with $\lfloor x\rfloor$ being the largest integer
less than or equal to $x$.
  Let
$$
\rho_j:=\delta_{c_j}\ast \frac{1}{\,R_j}\,\sum_{0\le k_2\le a_1 R_j-1}\,
\dots\sum_{0\le k_s\le a_s R_j-1}\,
\delta_{(0,-k_2/a_2,\dots,-k_1/a_s)}\ast \nu_\mu.
$$
As in the proof of (a), if $r>0$,  we have, letting $C=\sup_{t\in \mathbb{R}}\mu([-r,r]+t)<\infty$, that
$$
\rho_j([-r,r]^s)\le \frac{C}{\,R_j} \prod_{m=2}^s (4\,r\,a_m+1)\to 0,\quad j\to \infty.
$$
For the same reason, $\left(\delta_{e_1\,N_j/a_1}\ast \rho_j\right)([-r,r]^s)\to 0$ as $j\to \infty$.
Hence,
\begin{align*}
\left|\left(\delta_{a_1^{-1} e_1}-\delta_0\right)*\sigma_j\right|([-r,r]^s)&=
\left|\left(\delta_{N_j a_1^{-1} e_1}-\delta_0\right)*\rho_j\right|([-r,r]^s)\\
&\le \left(\delta_{N_j a_1^{-1} e_1}+\delta_0\right)*\rho_j([-r,r]^s)\to 0,
\end{align*}
as $j\to \infty$ and our claim follows when $l=a_i^{-1}\,e_i$, $i=1,\dots,s$,
since $r>0$ is arbitrary.
In general, if $l\in \mathcal{L}$, we can write
$$
\delta_l-\delta_0=\sum_{m=1}^s\,\tau_m\ast\left(\delta_{a_m^{-1} e_m}-\delta_0\right)
$$
where each $\tau_m$ is a finite sum of Dirac masses. Hence,
$$
|\left(\delta_l-\delta_0\right)*\sigma_j|(K)
=\left|\sum_{m=1}^s\,\tau_m\ast\left(\delta_{a_m^{-1} e_m}-\delta_0\right)*\sigma_j\right|(K)\to 0,
$$
as $j\to \infty$, for any compact set $K\subset  \mathbb{R}^s$, which proves our claim.
\end{proof}

If $c>0$ and $x$ is a real number,
we denote by $x\,(\text{mod}\,c)$ the unique real number $y$ in the interval
$[0,c)$ such that $x-y\in c\,\mathbb{Z}$.
The following theorem will be the key result to answer the questions raised at the beginning
of this section. We will only proof the equivalence of the upper-bounds inequalities
(i.e.~those involving the constant $B$) as the lower-bound ones
(involving the constant $A$ if $A>0$) can be obtained by very
similar techniques. We leave the details to the interested reader.

\begin{theorem}\label{frame1}
Let $a_1,\dots, a_s$ be $s$ positive real numbers linearly independent over $\mathbb{Q}$, with $s\ge 1$,
and let $G$ be the subgroup of $\mathbb{R}$ generated by $a_1,\dots, a_s$, i.e.
$$
G=\left\{\sum_{i=1}^s\,m_i\,a_i,\,\,m_i\in \mathbb{Z}\right\}.
$$
Let $\mu$ be a positive translation-bounded Borel measure on $\mathbb{R}$
and associate with it the  positive Borel measure $\nu_\mu$  on $\mathbb{R}^s$
defined by (\ref{nudef}). Let $A,B$ be real constants with $A\ge 0$ and $B>0$.
Then, the following are equivalent:

\begin{enumerate}[(a)]
\item
For any distinct $x_1,\dots,x_m\in G$
and any $c_1,\dots,c_m\in \mathbb{C}$, we have
$$
\limsup_{R\to \infty}\,\sup_{t\in \mathbb{R}}\,\frac{1}{R}\int_{[t, t+R]}\,
\big|\sum_{i=1}^m\,c_i\,e^{-2\pi i x_i\,\lambda}\big|^2\,d\mu(\lambda)\le\,B\,\sum_{i=1}^m\,|c_i|^2.
$$
and
$$
\liminf_{R\to \infty}\,\inf_{t\in \mathbb{R}}\,\frac{1}{R}\int_{[t, t+R]}\,
\big|\sum_{i=1}^m\,c_i\,e^{-2\pi i x_i\,\lambda}\big|^2\,d\mu(\lambda)\ge\,A\,\sum_{i=1}^m\,|c_i|^2.
$$

\item Any weak-$*$ limit $\sigma$ of a sequence extracted from the collection
$$
\delta_{(-t,\dots,-t)}\ast \frac{1}{a_1\,\dots\,a_s\,R}\,\sum_{0\le k_1\le a_1 R-1}\,
\dots,\sum_{0\le k_s\le a_s R-1}\,
\delta_{(-k_1/a_1,\dots,-k_1/a_s)}\ast \nu_\mu
$$
where $t\in \mathbb{R}$ and $R\to \infty$
is absolutely continuous with respect to the Lebesgue measure
and  with
Radon-Nikodym derivative
$$
\frac{d\sigma}{d\xi}=G\quad\text{satisfying}\,\,A\le G\le B \,\,\text{a.e.~on}\,\,\mathbb{R}^s.
$$
\item For any intervals $I_1\subset [0,1/a_1),\dots, I_s\subset [0,1/a_s)$,
we have
$$
\limsup_{R\to \infty}\,\sup_{t\in \mathbb{R}}\,\frac{1}{R\,a_1\dots a_s}\,
\mu\left(E(t, R, I_1,\dots,I_s)\right)
\le |I_1|\,\dots |I_s|\,B
$$
and
$$
\liminf_{R\to \infty}\,\inf_{t\in \mathbb{R}}\,\frac{1}{R\,a_1\dots a_s}\,
\mu\left(E(t, R, I_1,\dots,I_s)\right)
\ge |I_1|\,\dots |I_s|\,A,
$$
where
$$
E(t, R, I_1,\dots,I_s)=\left\{\lambda\in \mathbb{R}:\,\,
t\le \lambda\le t+R, \lambda\,({\rm mod}\,a_1^{-1})\in I_1,\dots,\lambda\,({\rm mod}\,a_s^{-1})\in I_s\right\}.
$$

\end{enumerate}
\end{theorem}
\begin{proof} As mentioned above we will only prove the case $A=0$ of the theorem.

Fix $x_1,\dots,x_m\in G$, with $x_i=\sum_{j=1}^s\,n_{ij}\,a_j$, $n_{ij}\in \mathbb{Z}$,
and let $c_1,\dots,c_m\in \mathbb{C}$. Note that
\begin{align*}
&\limsup_{R\to \infty}\,\sup_{t\in \mathbb{R}}\,\frac{1}{R}\,
\int_{[t, t+R]^s}\,
\big|\sum_{i=1}^m\,c_i\,e^{-2\pi i (\sum_{j=1}^s\,n_{ij}\,a_j\,\xi_j)}\big|^2\,d\nu_{\mu}(\xi_1,\dots,\xi_s)\\
&=\limsup_{R\to \infty}\,\sup_{t\in \mathbb{R}}\,\frac{1}{R}\int_{[t, t+R]}\,
\big|\sum_{i=1}^m\,c_i\,e^{-2\pi i x_i\,\lambda}\big|^2\,d\mu(\lambda):=L.
\end{align*}
Furthermore, replacing the set $[t,t+R]^s$ by the smaller set
$$
[t, t+N_1(R)\,a_1^{-1})\times\dots\times[t,t+ N_s(R)\,\,a_s^{-1})
$$
where $N_k(R)$ are the unique integers satisfying
$$
N_k(R)\,a_k^{-1}\le R<(N_k(R)+1)\,a_k^{-1},\quad k=1,\dots,s,
$$
does not change the first of the limits above.
Indeed, letting $I=[t,t+R]$ and $J_k=[t, t+N_k(R)\,a_k^{-1})$ for $k=1,\dots,s$,
we have
$$
I^s\setminus \prod_{k=1}^s\,J_k\subset \left[(I\setminus J_1)\times I^{s-1}\right]\cup
\left[I\times(I\setminus J_2)\times I^{s-2}\right]\cup\dots\cup \left[I^{s-1}\times(I\setminus J_s)\right]
:=\bigcup_{k=1}^s\,C_k.
$$
Letting $Q(\xi)=Q(\xi_1,\dots,\xi_s)=\sum_{i=1}^m\,c_i\,e^{-2\pi i (\sum_{j=1}^s\,n_{ij}\,a_j\,\xi_j)}$
we have
\begin{align*}
&\frac{1}{R}\,\int_{I^s\setminus \prod_{k=1}^s\,J_k}\,| Q|^2\,d\nu_{\mu}\le
\sum_{k=1}^s\,\frac{1}{R}\,\int_{C_k}\,| Q|^2\,d\nu_{\mu}\\
&\le\|Q\|_{\infty}^2\, \sum_{k=1}^s\,
\frac{1}{R}\,\int_{C_k}\,1\,d\nu_{\mu}=\|Q\|_{\infty}^2\, \sum_{k=1}^s\,
\frac{1}{R}\,\int_{I\setminus J_k}\,1\,d\mu\\
&\le \|Q\|_{\infty}^2\,\frac{1}{R}\,\sum_{k=1}^s\,\mu([t+N_k(R)\,a_k^{-1},t+(N_k(R)+1)\,a_k^{-1}])\to 0,
\,\,\text{as}\,\,R\to \infty,
\end{align*}
since $\mu$ is translation bounded.
We have thus
\begin{align*}
&L=\limsup_{R\to \infty}\,\sup_{t\in \mathbb{R}}\,\frac{1}{R}\,
\int_{\prod_{k=1}^s [t, t+N_k(R)\,\,a_k^{-1})}\,
\big|Q(\xi)\big|^2\,d\nu_{\mu}(\xi)\\
&=\limsup_{R\to \infty}\,\sup_{t\in \mathbb{R}}\,\frac{1}{R}\,\sum_{0\le k_1\le  a_1 R-1}
\,\dots\, \sum_{0\le k_s\le a_s R-1}\,
\int_{I(t,k_1,\dots,k_s)}\,
\big|Q(\xi)\big|^2\,d\nu_{\mu}(\xi).\\
\end{align*}
where
$$
I(t,k_1,\dots,k_s):=
[t+k_1\,a_1^{-1}, t+(k_1+1)\,a_1^{-1})\times\dots\times[t+k_s\,a_s^{-1},t+(k_s+1)\,a_s^{-1}).
$$
Letting
\begin{equation}\label{sigdef}
\sigma_{t,R}=\delta_{(-t,\dots,-t)}\ast \frac{1}{a_1\,\dots\,a_s\,R}\,\sum_{0\le k_1\le a_1 R-1}\,
\dots\,\sum_{0\le k_s\le a_s R-1}\,
\delta_{(-k_1/a_1,\dots,-k_1/a_s)}\ast \nu_\mu,
\end{equation}
and, using the periodicity of $Q$ with respect to the lattice
$\mathcal{L}:=\prod_{k=1}^s\,a_k^{-1}\,\mathbb{Z}$,
we have
\begin{equation}\label{L-result}
L=\limsup_{R\to \infty}\,\sup_{t\in \mathbb{R}}\,a_1\dots a_s\,
\int_{[0, a_1^{-1})\times\dots\times[0,a_s^{-1})}\,
\big|Q(\xi)\big|^2\,d\sigma_{t,R}(\xi)
\end{equation}
Using part (a) of  Lemma \ref{per}, with $c=(-t,\dots,-t)$,
it follows that the set
$$\{\sigma_{t,R}(K),\,\,t\in \mathbb{R},\,\,R\ge 1\}
$$ is bounded for
any compact subset $K$ of $\mathbb{R}^s$. Hence, any sequence extracted from the
collection of measures $\{\sigma_{t,R},\,\,t\in \mathbb{R},\,\,R\ge 1\}$ must have a
weak-$*$  convergent subsequence.
Furthermore, by part (b) of  Lemma \ref{per},
any weak-$*$ limit of a sequence $\sigma_{t_j,R_j}$, where $R_j\to \infty$,
must be periodic  with respect to the lattice $\mathcal{L}$ defined above.

If (b) holds, consider   sequences $\{t_j\}$ and $\{R_j\}$ with  $R_j\to \infty$
such that
$$
L=\lim_{j\to \infty}\,a_1\dots a_s\,
\int_{[0, a_1^{-1})\times[\dots\times0,a_s^{-1})}\,
\big|Q(\xi)\big|^2\,d\sigma_{t_j,R_j}(\xi).
$$
By weak-$*$ compactness, we can assume, by passing to a subsequence if necessary,
that  $\{\sigma_{t_j,R_j}\}$
is weak-$*$
convergent to a measure $\sigma$ as $j\to \infty$.
Using our hypothesis, $\sigma$ is absolutely continuous with
respect to the Lebesgue measure
and  with
Radon-Nikodym derivative $d\sigma/d\xi=G$ and $\|G\|_\infty\le B$. Using Lemma \ref{tau}
with $F=|Q|^2$, it  follows that
\begin{align*}
L&=a_1\dots a_s\,
\int_{[0, a_1^{-1})\times\dots\times[0,a_s^{-1})}\,
\big|Q(\xi)\big|^2\,G(\xi)\,d\xi\\
&\le B\,a_1\dots a_s\,
\int_{[0, a_1^{-1})\times\dots\times[0,a_s^{-1})}\,
\big|Q(\xi)\big|^2\,d\xi\\
&=B\,a_1\dots a_s\,
\int_{[0, a_1^{-1})\times\dots\times [0,a_s^{-1})}\,
\big|\sum_{i=1}^m\,c_i\,e^{-2\pi i (\sum_{j=1}^s\,n_{ij}\,a_j\,\xi_j)}\big|^2\,d\xi\\
&=B\,a_1\dots a_s\,\sum_{i,l=1}^m\,c_i\,\overline{c_l}\,
\int_{[0, a_1^{-1})\times\dots\times[0,a_s^{-1})}\,
e^{-2\pi i \sum_{j=1}^s\,(n_{ij}-n_{lj})\,a_j\,\xi_j}\,d\xi\\
&=B\,a_1\dots a_s\,\sum_{i,l=1}^m\,c_i\,\overline{c_l}\,
\prod_{j=1}^s\,
\int_{[0, a_j^{-1})}\,
e^{-2\pi i \,(n_{ij}-n_{lj})\,a_j\,\xi_j}\,d\xi_j\\
&=B\,\sum_{i=1}^m\,|c_i|^2,
\end{align*}
showing that (a) holds. Note that, in the last step of the previous computation, we used the
fact that $n_{ij}=n_{lj}$ for all $j=1,\dots,m$, implies that $x_i=x_l$ (using the
linear independence of $a_1,\dots, a_s$ over $\mathbb{Q}$) and thus that $i=l$,
since the $x_i$'s are assumed to be distinct.

Conversely, if (a) holds and $\sigma$ is a  weak-$*$ limit
of a sequence $\{\sigma_{t_j,R_j}\}$, with $R_j\to \infty$ as $j\to \infty$,
then $\sigma$ is periodic  with respect to the lattice $\mathcal{L}$ and
we have by the computation above, that
$$
a_1\dots a_s\,
\int_{[0, a_1^{-1})\times\dots\times[0,a_s^{-1})}\,
\big|Q(\xi)\big|^2\,d\sigma(\xi)
\le B\,a_1\dots a_s\,\int_{[0, a_1^{-1})\times\dots\times[0,a_s^{-1})}\,
\big|Q(\xi)\big|^2\,d\xi
$$
for any trigonometric polynomial $Q(\xi)=\sum_{i=1}^m\,c_i\,
e^{-2\pi i (\sum_{j=1}^s\,n_{ij}\,a_j\,\xi_j)}$.
Since the space of such trigonometric polynomials is dense (with respect to the
sup-norm) in the space of continuous functions
which are periodic with respect to the lattice $\mathcal{L}$,
we have
$$
\int_{[0, a_1^{-1})\times\dots\times[0,a_s^{-1})}\,
\sum_{l\in \mathcal{L}}\phi(\xi+l)\,d\sigma(\xi) \le B\,
\int_{[0, a_1^{-1})\times\dots\times[0,a_s^{-1})}\,
\sum_{l\in \mathcal{L}}\phi(\xi+l)\,d\xi
$$
for any compactly supported continuous function $\phi\ge 0$ on  $\mathbb{R}^s$
and thus, since $\sigma$ is $\mathcal{L}$-periodic,
$$
\int_{\mathbb{R}^s}\,\phi(\xi)\,d\sigma(\xi) \le B\,\int_{\mathbb{R}^s}\,\phi(\xi)\,d\xi.
$$
Standard arguments show that $\sigma$ must  absolutely continuous with respect to the Lebesgue measure
and  with a
Radon-Nikodym derivative $G\in L^\infty(\mathbb{R}^s)$ satisfying $\|G\|_\infty\le B$, which shows that (b) holds.
Thus the statements (a) and (b) are equivalent.
Now, consider intervals $I_j\subset [0,1/a_j)$ for $1\le j\le s$
and define the sets $F_j=\bigcup_{k\in \mathbb{Z}}\,I_j+k/a_j$ and let $F=\prod_{j=1}^s\,F_j\subset\mathbb{R}^s$.
Note that
$$
\left\{\lambda\in \mathbb{R}:\,\,
t\le \lambda\le t+R, \lambda\,({\rm mod}\, a_1^{-1})\in I_1,\dots, \lambda\,({\rm mod}\, a_s^{-1})\in I_s\right\}
=[t,t+R]\cap F_1\cap\dots \cap F_s.
$$
and that the function
$$
\chi_F(\xi)=\chi_{F_1}(\xi_1)\,\dots \chi_{F_s}(\xi_s),\quad \xi=(\xi_1,\dots,\xi_s)\in \mathbb{R}^s
$$
is $\mathcal{L}$-periodic. Let
\begin{align*}
M&:=\limsup_{R\to \infty}\,\sup_{t\in \mathbb{R}}\,\frac{1}{R}\,
\mu\left(\left\{\lambda\in \mathbb{R}:\,\,
t\le \lambda\le t+R, \lambda\,({\rm mod}\, a_1^{-1})\in I_1,\dots, \lambda\,({\rm mod}\, a_s^{-1})\in I_s\right\}\right)\\
&=\limsup_{R\to \infty}\,\sup_{t\in \mathbb{R}}\,\frac{1}{R}\int_{[t, t+R]}\,
\chi_{F_1}(\lambda)\,\dots \chi_{F_s}(\lambda)
\,d\mu(\lambda)\\
&=\limsup_{R\to \infty}\,\sup_{t\in \mathbb{R}}\,\frac{1}{R}\,
\int_{[t, t+R]^s}\,\chi_{F_1}(\xi_1)\,\dots \chi_{F_s}(\xi_s)
\,d\nu_{\mu}(\xi_1,\dots,\xi_s)\\
&=\limsup_{R\to \infty}\,\sup_{t\in \mathbb{R}}\,\frac{1}{R}\,
\int_{[t, t+R]^s}\,\chi_{F}(\xi)\,d\nu_{\mu}(\xi).
\end{align*}
By a  computation similar to the one done to obtain (\ref{L-result})
 (with $\chi_{F}(\xi)$ replacing $|Q(\xi)|^2$)
we obtain, using the $\mathcal{L}$-periodicity of $\chi_{F}(\xi)$, that
$$
M=\limsup_{R\to \infty}\,\sup_{t\in \mathbb{R}}\,a_1\dots a_s\,
\int_{[0, a_1^{-1})\times\dots\times[0,a_s^{-1})}\,\chi_{F}(\xi)
\,d\sigma_{t,R}(\xi)
$$
with $\sigma_{t,R}$ as in (\ref{sigdef}). Let $\sigma_{t_j,R_j}$ a sequence with $R_j\to \infty$
such that
$$
M=\lim_{j\to \infty}\,a_1\dots a_s\,
\int_{[0, a_1^{-1})\times\dots\times[0,a_s^{-1})}\,\chi_{F}(\xi)
\,d\sigma_{t_j,R_j}(\xi).
$$
By a weak-$*$ compactness argument, we can assume
that ${\sigma_{t_j,R_j}}$ converges in the weak-$*$ topology to a $\mathcal{L}$-periodic measure $\sigma$.
If (b) holds, we can use Lemma \ref{tau}
applied to the sequence $\{\sigma_{t_j,R_j}\}$ to show that
$$
M=a_1\dots a_s\,
\int_{[0, a_1^{-1})\times\dots\times[0,a_s^{-1})}\,\chi_{F}(\xi)
\,G(\xi)\,d\xi.
$$
We have thus
$$
M\le B\,a_1\dots a_s\,|I_1|\dots |I_s|,
$$
which shows that (c) holds. Conversely, if (c)  holds, then for any intervals $I_j$ with
$I_j\subset [0,1/a_j)$ for $1\le j\le s$, we have
$$
\limsup_{R\to \infty}\,\sup_{t\in \mathbb{R}}\,
\int_{I_1\times\dots\times I_s}\,1
\,d\sigma_{t,R}(\xi)\le B\,|I_1|\dots |I_s|.
$$
Let $H_r= I_1^r\times\dots\times I_s^r$, for $1\le r\le R$, where, for each
$r$, $I_j^r$ is an interval contained in $[0,1/a_j)$.
If $H_r\cap H_s=\emptyset$ when $r\ne s$ and
 $h=\sum_{r=1}^R\,c_r\,\chi_{H_r}$ with $c_r\ge 0$,
we have
$$
\limsup_{R\to \infty}\,\sup_{t\in \mathbb{R}}\,
\int_{\mathbb{R}^s}\,h(\xi)
\,d\sigma_{t,R}(\xi)\le B\,\sum_{r=1}^R\,c_r\,|I_1^r|\dots |I_s^r|
=B\,\int_{\mathbb{R}^s}\,h(\xi)\,d\xi.
$$
Suppose that $\sigma$ is a weak-$*$ limit of a sequence ${\sigma_{t_j,R_j}}$ with $R_j\to \infty$
and let $\phi\ge 0$ be a compactly supported continuous function of $\mathbb{R}^s$.
We have
$$
\int_{\mathbb{R}^s}\,\phi(\xi)\,d\sigma(\xi)=\lim_{j\to \infty}\,\int_{\mathbb{R}^s}\,\phi(\xi)
\,d\sigma_{t_j,R_j}(\xi)=
\lim_{j\to \infty}\,\sum_{l\in \mathcal{L} }\,\int_{[0, a_1^{-1})\times\dots\times[0,a_s^{-1})-l}\,
\phi(\xi)
\,d\sigma_{t_j,R_j}(\xi)
$$
where only a finite number of terms are non-zero in the last series, since $\phi$ is compactly supported.
Hence,
\begin{align*}
&\int_{\mathbb{R}^s}\,\phi(\xi)\,d\sigma(\xi)=
\lim_{j\to \infty}\,\sum_{l\in \mathcal{L} }\,
\int_{[0, a_1^{-1})\times\dots\times[0,a_s^{-1})}\,
\phi(\xi-l)
\,d\left(\delta_{l}*\sigma_{t_j,R_j}\right)(\xi)
\\
&=\lim_{j\to \infty}\,\sum_{l\in \mathcal{L} }\,
\int_{[0, a_1^{-1})\times\dots\times[0,a_s^{-1})}\,
\phi(\xi-l)
\,d\sigma_{t_j,R_j}(\xi)\\
&+\lim_{j\to \infty}\,\sum_{l\in \mathcal{L} }\,
\int_{[0, a_1^{-1})\times\dots\times[0,a_s^{-1})}\,
\phi(\xi-l)
\,d\nu_j(\xi)
\end{align*}
where $\nu_j:=(\delta_{l}-\delta_0)*\sigma_{t_j,R_j}$.
Using part (b) of Lemma \ref{per} and Lemma  \ref{tau}, the second limit above must be zero.

Hence,
$$
\int_{\mathbb{R}^s}\,\phi(\xi)\,d\sigma(\xi)=\lim_{j\to \infty}\,
\int_{[0, a_1^{-1})\times\dots\times[0,a_s^{-1})}\,
\psi(\xi)\,d\sigma_{t_j,R_j}(\xi),
$$
where
$$
\psi(\xi)=\sum_{l\in \mathcal{L}}\,
\phi(\xi+l),\quad \xi\in \mathbb{R}^s,
$$
is continuous and $\mathcal{L}$-periodic. Since the restriction of $\psi$ to
the set $[0, a_1^{-1})\times\dots\times[0,a_s^{-1})$ is uniformly continuous, we can find,
for any $\epsilon>0$, disjoints sets $H_r$, $r=1,\dots,R$, as above and constants
$c_r,d_r\ge 0$ such that
if $h_1=\sum_{r=1}^R\,c_r\,\chi_{H_r}$ and $h_2=\sum_{r=1}^R\,d_r\,\chi_{H_r}$,
we have  $h_1\le \psi \le h_2$ and
$$
h_2-h_1\le \epsilon
$$
on the set $[0, a_1^{-1})\times\dots\times[0,a_s^{-1})$.
We have thus,

\begin{align*}
&\int_{\mathbb{R}^s}\,\phi(\xi)
\,d\sigma(\xi)=\lim_{j\to \infty}\,\int_{[0, a_1^{-1})\times\dots\times[0,a_s^{-1})}\,
\psi(\xi)\,d\sigma_{t_j,R_j}(\xi)\\
&\le \limsup _{j\to \infty}\,\int_{[0, a_1^{-1})\times\dots\times[0,a_s^{-1})}\,
h_2(\xi)\,d\sigma_{t_j,R_j}(\xi)\le B\,\int_{[0, a_1^{-1})\times\dots\times[0,a_s^{-1})}\,h_2(\xi)\,d\xi\\
&\le B\,\int_{[0, a_1^{-1})\times\dots\times[0,a_s^{-1})}\,\left(h_1(\xi)+\epsilon\right)\,d\xi
\le  B\,\int_{[0, a_1^{-1})\times\dots\times[0,a_s^{-1})}\,\psi(\xi)\,d\xi
+B\,a_1^{-1}\,\dots a_s^{-1}\,\epsilon.
\end{align*}

Since $\epsilon>0$ is arbitrary, it follows that
$$
\int_{\mathbb{R}^s}\,\phi(\xi)
\,d\sigma(\xi)\le  B\,\int_{[0, a_1^{-1})\times\dots\times[0,a_s^{-1})}\,\psi(\xi)\,d\xi
= B\,\int_{\mathbb{R}^s}\,\phi(\xi)\,d\xi.
$$
In particular, we have that, for any complex-valued compactly supported continuous function on  $\mathbb{R}^s$,
$$
\int_{\mathbb{R}^s}\,|\phi(\xi)|
\,d\sigma(\xi)\le  B\,\int_{\mathbb{R}^s}\,|\phi(\xi)|\,d\xi.
$$
Standard arguments show that $\sigma$ must be  absolutely continuous with respect to the Lebesgue measure
and  with a
Radon-Nikodym derivative $G$ satisfying $\|G\|_{\infty}\le B$, proving (b).
\end{proof}

\begin{corollary}\label{tightframe} Under the same assumptions as Theorem \ref{frame1}, the following
statements are equivalent:
\begin{enumerate}[(a)]
\item
There exists a positive constants $A$ such that, for any distinct $x_1,\dots,x_m\in G$
and any $c_1,\dots,c_m\in \mathbb{C}$, we have
$$
\lim_{R\to \infty}\,\frac{1}{R}\int_{[t, t+R]}\,
\big|\sum_{i=1}^m\,c_i\,e^{-2\pi i x_i\,\lambda}\big|^2\,d\mu(\lambda)=A\,\sum_{i=1}^m\,|c_i|^2,
$$
uniformly for $t\in \mathbb{R}$.

\item There exists a constant  $A>0$
 such that any weak-$*$ limit $\sigma$ of a sequence extracted from the collection
$$
\delta_{(-t,\dots,-t)}\ast \frac{1}{a_1\,\dots\,a_s\,R}\,\sum_{0\le k_1\le a_1 R-1}\,
\dots,\sum_{0\le k_s\le a_s R-1}\,
\delta_{(-k_1/a1,\dots,-k_1/a_s)}\ast \nu_\mu
$$
is equal to the absolutely continuous measure $d\sigma=A\,d\xi$, where $d\xi$ represents the Lebesgue
measure on $\mathbb{R}^s$.
\item There exists a constant $A>0$ such that,
for any intervals $I_1\subset [0,1/a_1),\dots, I_s\subset [0,1/a_s)$,
we have
$$
\lim_{R\to \infty}\,\frac{1}{R\,a_1\dots a_s}\,
\mu\left(E(t, R, I_1,\dots,I_s)\right)
=|I_1|\,\dots |I_s|\,A,
$$
uniformly for $t\in \mathbb{R}$,
where
$$
E(t, R, I_1,\dots,I_s)=\left\{\lambda\in \mathbb{R}:\,\,
t\le \lambda\le t+R, \lambda\,({\rm mod}\, a_1^{-1})\in I_1,\dots, \lambda\,({\rm mod}\, a_s^{-1})\in I_s\right\}.
$$
\end{enumerate}
\end{corollary}
If $G$ is a subgroup of $\mathbb{R}$, we will denote by $\Pi_G$ the set of trigonometric polynomials
$P(\lambda)$ on $\mathbb{R}$
with spectrum in $G$, i.e. those of the form
$$
P(\lambda)=\sum_{i=1}^m\,c_i\,e^{-2\pi i x_i \lambda},\quad x_i\in G,\,\,c_i\in \mathbb{C}.
$$
The mean of $P$ is defined to be
$$
\mathcal{M}(P)=\lim_{R\to \infty}\,\frac{1}{R}\int_{[-R/2, R/2]}\,P(\lambda)\,d\lambda.
$$
Note that $\mathcal{M}(|P|^2)|=\sum_{i=1}^m\,|c_i|^2$, if $P(\lambda)$ is as above.
Combining these results with those of the previous sections,
we obtain the following characterizations.

\begin{theorem}\label{main2}
Let $a_1,\dots, a_s$ be $s$ positive real numbers linearly independent over $\mathbb{Q}$, with $s\ge 1$,
and let $G$ be the subgroup of $\mathbb{R}$ generated by $a_1,\dots, a_s$, i.e.
$$
G=\left\{\sum_{i=1}^s\,m_i\,a_i,\,\,m_i\in \mathbb{Z}\right\}.
$$
Let $\mu$ be a positive translation-bounded Borel measure on $\mathbb{R}$.
If $I_k$ are intervals with $I_k\subset [0,1/a_k)$, $k=1,\dots s$,
let $E(t, R, I_1,\dots,I_s)$ denote the set
$$
\left\{\lambda\in \mathbb{R}:\,\,
t\le \lambda\le t+R, \lambda\,({\rm mod}\, a_1^{-1})\in I_1,\dots, \lambda\,({\rm mod}\, a_1^{-1})\in I_s\right\}.
$$
\begin{enumerate}[(a)]
\item The measure $\mu\in \mathcal{B}(G,B)$ if and only if
any of the two following statement holds:
\begin{enumerate}[(i)]
\item For any intervals $I_k\subset [0,1/a_k)$, $k=1,\dots s$,
we have
$$
\limsup_{R\to \infty}\,\sup_{t\in \mathbb{R}}\,\frac{1}{R\,a_1\dots a_s}\,
\mu\left(E(t, R, I_1,\dots,I_s)\right)
\le |I_1|\,\dots |I_s|\,B.
$$
\item  We have
$$
\mathcal{D}^+(|P|^2\,\mu)\le B\,\mathcal{M}(|P|^2),\quad P\in \Pi_G.
$$
\end{enumerate}
\item The measure $\mu\in \mathcal{F}(G,A,B)$ if and only if
any of the two following statement holds:
\begin{enumerate}[(i)]
\item For any intervals $I_k\subset [0,1/a_k)$, $k=1,\dots s$,
we have
$$
\liminf_{R\to \infty}\,\inf_{t\in \mathbb{R}}\,\frac{1}{R\,a_1\dots a_s}\,
\mu\left(E(t, R, I_1,\dots,I_s)\right)
\ge |I_1|\,\dots |I_s|\,A.
$$
and
$$
\limsup_{R\to \infty}\,\sup_{t\in \mathbb{R}}\,\frac{1}{R\,a_1\dots a_s}\,
\mu\left(E(t, R, I_1,\dots,I_s)\right)
\le |I_1|\,\dots |I_s|\,B.
$$
\item  We have
$$
A\,\mathcal{M}(|P|^2)\le \mathcal{D}^-(|P|^2\,\mu)\le
\mathcal{D}^+(|P|^2\,\mu)\le B\,\mathcal{M}(|P|^2),\quad P\in \Pi_G.
$$
\end{enumerate}
\item The measure $\mu\in \mathcal{F}(G,A,A)$ if and only if
any of the two following statement holds:
\begin{enumerate}[(i)]
\item For any intervals $I_k\subset [0,1/a_k)$, $k=1,\dots s$,
we have
$$
\lim_{R\to \infty}\,\frac{1}{R\,a_1\dots a_s}\,
\mu\left(E(t, R, I_1,\dots,I_s)\right)
= |I_1|\,\dots |I_s|\,A.
$$
uniformly for $t\in \mathbb{R}$.
\item We have
$$
\mathcal{D}(|P|^2\,\mu)=A\,\mathcal{M}(|P|^2),\quad P\in \Pi_G.
$$
\end{enumerate}
\end{enumerate}
\end{theorem}
\begin{proof} The proof of (a) and (b) follow immediately from the equivalence of conditions
(a) and (c) in Theorem \ref{frame1} together with Theorem \ref{th1.2} and Theorem \ref{th1.1},
respectively. The statement in (c) is an immediate consequence of Corollary \ref{tightframe}
and the case $A=B$ of Theorem \ref{th1.1}.
\end{proof}
In the case where $G$ is the discrete subgroup $G=a\,\mathbb{Z}$, with $a>0$,
the conditions
given in the previous theorem, more particularly the ones given in statement (c), are
strongly related to the notions of ``equidistributed'' sequence or, more specifically, to that
of ``well-distributed'' sequence of real numbers. We will deal with bi-infinite sequences,
so the definition given below is slightly different than the classical one dealing with one-sided
sequences (see \cite{KN}).
\begin{definition}
If $b>0$,  we call
a sequence of real numbers $\{\lambda_n\}_{n\in \mathbb{Z}}$ well-distributed modulo $b$
if, for every
interval $I\subset [0,b)$ and every integer $M\in \mathbb{Z}$,
$$
\lim_{N\to \infty}\,\frac{\#\{n,\,\,\lambda_n\,({\rm mod}\, b)\in I,\,\,M\le n\le M+N-1\}}{N}=|I|/b
$$
 uniformly for $M\in \mathbb{Z}$.
\end{definition}
The condition (c) of Theorem \ref{main2} can be rephrased using the notion of well-distributed
sequences when dealing with measures of the form $\mu=\delta_\Lambda$, where $\Lambda$ is a discrete
subset of $\mathbb{R}$ (i.e. the intersection of $\Lambda$ with any compact set is finite).
 \begin{corollary}\label{well-dist}
Let $G= a\,\mathbb{Z}$ where $a>0$. Suppose that $\Lambda$ is a discrete subset of $\mathbb{R}$
and consider a sequence $\{\lambda_n\}_{n\in \mathbb{Z}}$ enumerating the elements of $\Lambda$ in such a way
that
$$
\lambda_n< \lambda_{n+1},\quad n\ge 1.
$$
Then, $\delta_\Lambda\in \mathcal{F}(G,A,A)$ if and only if
\begin{enumerate}[(i)]
\item
$\mathcal{D}(\Lambda)=A$.
\item
The sequence   $\{\lambda_n\}_{n\in \mathbb{Z}}$ is well-distributed modulo $a^{-1}$.
\end{enumerate}
\end{corollary}
\begin{proof} Using part (c) of Theorem \ref{main2}, we easily see that
$\delta_\Lambda\in \mathcal{F}(G,A,A)$ if and only if, for any
interval $I\subset [0,a^{-1})$, we have
\begin{equation}\label{lim1}
\lim_{R\to \infty}\,\frac{1}{R}\,
\#\left\{\lambda\in \Lambda,\,\,t\le \lambda\le t+R,\,\,\lambda\,({\rm mod}\, a^{-1})\in I \right\}
= a\,|I|\,A,
\end{equation}
uniformly for $t\in \mathbb{R}$.
In particular, if (i) holds, we obtain, taking $I=[0,a^{-1})$, that $\mathcal{D}(\Lambda)=A$.

If $M\in \mathbb{Z}$, the set $A_{M,N}:=\{\lambda_n,\,\,M\le n\le M+N-1\}
=[\lambda_M,\lambda_{M+N-1}] \cap \Lambda$.
Note that for each $M$, $\lambda_{M+N-1}-\lambda_M\to \infty$ as $N\to \infty$,
uniformly for $M\in \mathbb{Z}$. Otherwise, we could find a number $L>0$
and intervals $I_N$ of length
bounded by $L$ containing at least $N$ elements of $\Lambda$, for any $N\ge 1$,
which would imply that $\mathcal{D}^+(\Lambda)=\infty$.
Hence, for any
interval $I\subset [0,a^{-1})$, we have
\begin{align*}
&\frac{\#\{n,\,\,\lambda_n\,({\rm mod}\,a^{-1})\in I,\,\,M\le n\le M+N\}}{N}\\
&=\frac{\#\{\lambda\in \Lambda,\,\,\lambda_M\le \lambda\le\lambda_{M+N-1},
\,\,\lambda\,({\rm mod}\,a^{-1})\in I \}}{\lambda_{M+N-1}-\lambda_M}\,
\left(\frac{\lambda_{M+N-1}-\lambda_M}{\#A_{M,N}}\right)
\end{align*}
Since $\mathcal{D}(\Lambda)=A$, we have $\lim_{N\to \infty}\,\frac{\#A_{M,N}}{\lambda_{M+N-1}-\lambda_M}=A$,
uniformly for $M\in \mathbb{Z}$.
Furthermore, using (\ref{lim1}), we obtain
$$
\lim_{N\to \infty}\,\frac{\#\{\lambda\in \Lambda,\,\,\lambda_M\le \lambda\le\lambda_{M+N-1},
\,\,\lambda\,({\rm mod}\,a^{-1})\in I \}}{\lambda_{M+N-1}-\lambda_M}
= a\,|I|\,A,
$$
uniformly for $M\in \mathbb{Z}$
as $N\to \infty$. We deduce that
$$
\lim_{N\to \infty}\,\frac{\#\{n,\,\,\lambda_n\,({\rm mod}\,a^{-1})\in I,\,\,M\le n\le M+N\}}{N}= a\,|I|,
$$
uniformly for $M\in \mathbb{Z}$, proving (ii). Conversely, if (i) and (ii) hold,
we have, for any $t\in \mathbb{R}$ and any $R>0$ large enough, that
$$
\left\{\lambda\in \Lambda,\,\,t\le \lambda\le t+R \right\}=
\{\lambda_n,\,\,M(t,R)\le n\le M(t,R)+N(t,R)-1\},
$$
for some $M(t,R)\in \mathbb{Z}$ and $N(t,R)\ge 1$. Furthermore, using (i), we have
$$
N(t,R)/R\to A, \quad R\to \infty,
$$
uniformly for $t\in \mathbb{R}$
and using (ii), we have
$$
\lim_{R\to \infty}\,\frac{1}{N(t,R)}\,
\#\left\{\lambda_n,\,\,M(t,R)\le n\le M(t,R)+N(t,R)-1,\,\,\lambda_n\,({\rm mod}\,a^{-1})\in I \right\}=|I|\,a,
$$
uniformly for $t\in \mathbb{R}$.
 Hence, for any
interval $I\subset [0,a^{-1})$, we have
\begin{align*}
&\lim_{R\to \infty}\,\frac{1}{R}\,
\#\left\{\lambda\in \Lambda,\,\,t\le \lambda\le t+R,\,\,\lambda\,({\rm mod}\,a^{-1})\in I \right\}\\
&= \lim_{R\to \infty}\,\frac{1}{R}\,
\#\left\{\lambda_n,\,\,M(t,R)\le n\le M(t,R)+N(t,R)-1,\,\,\lambda_n\,({\rm mod}\,a^{-1})\in I \right\}
&=|I|\,a\,A,
\end{align*}
showing that (\ref{lim1}) holds and thus that $\delta_\Lambda\in \mathcal{F}(G,A,A)$.
\end{proof}
Note that the condition (i) in the previous result is essential. For example, we can easily
construct a discrete set $\Lambda$ with $\mathcal{D}^+(\Lambda)=0$ such that the associated sequence
$\{\lambda_n\}$ defined in the previous corollary is well-distributed modulo $1$.

When the group $G$ is generated by at least two linearly independent (over $\mathbb{Q}$)
elements, it must be dense in  $\mathbb{R}$ and the conditions given in Theorem \ref{main2}
become more difficult to satisfy.
However, it is easy to check that the statement (a) in Corollary \ref{tightframe}
holds for any finitely generated subgroup $G$ if $d\mu=d\lambda$, the Lebesgue measure on $\mathbb{R}$.
It follows therefore that if $a_1,\dots, a_s$ are real numbers linearly
independent over $\mathbb{Q}$ and  $I_1,\dots,I_s$ are intervals with $I_j\subset [0,1/a_j)$,
$j=1,\dots,s$, then
$$
\lim_{R\to \infty}\,\frac{1}{R\,a_1\dots a_s}\,
|E(t, R, I_1,\dots,I_s)|
=|I_1|\,\dots |I_s|
$$
uniformly for $t\in \mathbb{R}$.
This can be interpreted, in the language of probability
theory, as saying that the events of
belonging to the intervals
 $I_j$ modulo $1/a_j$, $j=1,\dots,s$, are asymptotically independent.
For any finitely generated subgroup $G$, one can also construct discrete measures
in $\mathcal{F}(G,A,A)$. In fact, lattices will yield such measures
as long as $A$ does not belong to the $\mathbb{Q}$-linear span of $G$.
\begin{proposition}\label{lattice}
Let $a_1,\dots, a_s$ be $s$ positive real numbers linearly independent over $\mathbb{Q}$, with $s\ge 1$,
and let $G$ be the subgroup of $\mathbb{R}$ generated by $a_1,\dots, a_s$, i.e.
$$
G=\left\{\sum_{i=1}^s\,m_i\,a_i,\,\,m_i\in \mathbb{Z}\right\}.
$$
Let $b>0$ and  $\mu=\delta_\Lambda$, where $\Lambda=b\,\mathbb{Z}$.
Then, there exist constants $A,B>0$ such that $\mu\in \mathcal{F}(G,A,B)$
if and only if $1/b \notin \text{span}_{\mathbb{Q}}(a_1,\dots, a_m)$
and, in that case, we can take $A=B=1/b$.
\end{proposition}
\begin{proof} If $b^{-1}\in \text{span}_{\mathbb{Q}}(a_1,\dots, a_m)$, there exists an integer $l\ge 1$ such that
$l\,b^{-1}\in G$. Hence $k\,l\,b^{-1}\in G$ for any integer $k$, and if $c_0,\dots,c_M\in \mathbb{C}$, we have
$$
\frac{1}{R}\,\int_{[t,t+R]}\,
\left|\sum_{k=0}^M\,c_k\,e^{-2\pi i kl  b^{-1}\lambda}\right|^2\,d\mu(\lambda)=
\big|\sum_{k=0}^M\,c_k\big|^2\,\mu([t,t+R])/R\to 1/b\, \big|\sum_{k=0}^M\,c_k\big|^2
$$
uniformly for $t\in \mathbb{R}$ as $R\to \infty$.
Taking $c_k=1$, for all $k$, yields
$$|\sum_{k=0}^M\,c_k\big|^2=(M+1)\,\sum_{k=0}^M\,|c_k|^2,
$$
which shows that $\mu\notin \mathcal{B}(G,B)$ for any $B>0$.

If $1/b \notin \text{span}_{\mathbb{Q}}(a_1,\dots, a_m)$, then $x\,k\,b\notin \mathbb{Z}$
for any $k\in \mathbb{Z}\setminus \{0\}$ and any $x\in G\setminus \{0\}$.
Hence if $x_1,\dots,x_m$ are distinct elements of $G$ and $c_1,\dots,c_m\in  \mathbb{C}$, we have
$$
\frac{1}{R}\,\int_{[t,t+R]}\,
\left|\sum_{j=1}^m\,c_j\,e^{-2\pi i x_j\lambda}\right|^2\,d\mu(\lambda)=
\sum_{j,l=1}^m\,c_j\,\overline{c_l}\,\frac{1}{R}\,\sum_{t/b\le k\le (t+R)/b}
\,e^{-2\pi i (x_j-x_l) k b}\to \frac{1}{b}\,\sum_{j=1}^m\,|c_j|^2,
$$
uniformly for $t\in \mathbb{R}$, as $R\to \infty$, showing that $\mu\in \mathcal{F}(G,1/b,1/b)$.
\end{proof}
Recall (see \cite{Ka, Me1}) that
a function $F$ defined on the real line is called (Bohr) \emph{almost-periodic} if
it is continuous and
for every $\epsilon>0$
there exists a number $\Lambda=\Lambda(\epsilon,F)>0$ such that every interval of length
$\Lambda$ contains a number $\tau$ such that
\begin{equation}\label{almost-period}
\sup_{x\in \mathbb{R}}\,|F(x-\tau)-F(x)|<\epsilon.
\end{equation}
A number  $\tau$ such that (\ref{almost-period}) holds is called an {\it $\epsilon$-almost period}
of $F$. The space of almost-periodic functions on $\mathbb{R}$ can be characterized as
the sup-norm closure of the space of trigonometric polynomials $P(\lambda)$
associated with arbitrary real frequencies, i.e. functions of the form
$$
P(\lambda)=\sum_{i=1}^m\,c_i\,e^{-2\pi i x_i\,\lambda},\quad c_i\in \mathbb{C},\,\,x_i\in \mathbb{R}.
$$
If  $F(\lambda)$ is almost-periodic, the \emph{mean-value} of $F$, $\mathcal{M}(F)$, defined by
$$
\mathcal{M}(F)=\lim_{R\to \infty}\,\frac{1}{R}\int_{[-R/2, R/2]}\,F(\lambda)\,d\lambda
$$
exists and the spectrum of $F$ consists of all the real numbers $x$ such that
$$
\mathcal{M}(F(\lambda)\,e^{-2\pi i x\,\lambda})\ne 0.
$$
If the spectrum of $F$ is contained in the subgroup $G$, then $F$ can be uniformly
approximated arbitrary closely by trigonometric polynomials with spectrum in $G$ (see \cite{Me1}).
Note that if $F(\lambda)$ is almost-periodic and a trigonometric polynomial $P(\lambda)$
satisfies $\|F-P\|_\infty\le \epsilon$, then for any positive translation-bounded Borel measure $\mu$
on $\mathbb{R}$, we have
$$
\frac{1}{R}\int_{[t, t+R]}\,
\big|F(\lambda)-P(\lambda)\big|^2\,d\mu(\lambda)\le \epsilon^2\,\frac{1}{R}\,\mu\left([t, t+R]\right),\quad
t\in \mathbb{R},\,\,R>0,
$$
and, in particular,
$$
\limsup_{R\to \infty}\,\sup_{t\in \mathbb{R}}\,\frac{1}{R}\int_{[t, t+R]}\,
\big|F(\lambda)-P(\lambda)\big|^2\,d\mu(\lambda)\le \epsilon^2\,\mathcal{D}^+(\mu).
$$
It follows immediately that any of the inequalities satisfied by the class of trigonometric
polynomials with spectrum in the subgroup $G$ and used to characterize
the measures in $\mathcal{B}(G,B)$
or $\mathcal{F}(G,A,B)$ must also be satisfied by the functions in $AP(G)$, the collection of
almost-periodic functions on $\mathbb{R}$ having spectrum contained in $G$.
We have thus the following.
\begin{theorem}\label{a-p} Let $G$ be a subgroup of $\mathbb{R}$ and let $\mu$ be a positive Borel
measure on $\mathbb{R}$. Then,
\begin{enumerate}[(a)]
\item $\mu$ belongs to $\mathcal{B}(G,B)$ if and only if
$$
\mathcal{D^+}\left(|F|^2\,\mu\right)\le B\,\mathcal{M}(|F|^2),\quad F\in AP(G).
$$
\item $\mu$ belongs to $\mathcal{F}(G,A, B)$ if and only if
$$
A\,\mathcal{M}(|F|^2)\le \mathcal{D^-}\left(|F|^2\,\mu\right)
\le \mathcal{D^+}\left(|F|^2\,\mu\right)\le B\,\mathcal{M}(|F|^2),\quad F\in AP(G).
$$
\item $\mu$ belongs to $\mathcal{F}(G,A,A)$ if and only if
$$
\mathcal{D}\left(|F|^2\,\mu\right)=A\,\mathcal{M}(|F|^2),\quad F\in AP(G).
$$
\end{enumerate}

\end{theorem}

\section{Existence of discrete measures in $\mathcal{F}(\mathbb{R},A,B)$}
As we mentioned earlier, the Lebesgue measure on $\mathbb{R}$, $d\mu=d\lambda$,
belongs to $\mathcal{F}(\mathbb{R},1,1)$. On the other hand, we do not know a single
explicit example of a discrete set $\Lambda\subset \mathbb{R}$ such that the associated
measure $\mu=\delta_\Lambda$ belongs to $\mathcal{F}(\mathbb{R},A,B)$ for some constants
$0<A\le B<\infty$. The fact that the construction of such a set should be extremely difficult
is pretty clear
by considering the conditions that the sets $E(t,R, I_1,\dots,I_s)$ need to satisfy
in part (b) of Theorem \ref{main2} and this for any choice of numbers $a_1,\dots, a_s$
linearly independent over $\mathbb{Q}$.
In this last section, our goal will be to prove the existence of such a set.

\medskip

It might seem, a priori, that simple quasicrystals could yield an answer to the problem above
as they have been shown to be universal sampling sets my B.~Matei and Y.~Meyer in \cite{MM}.
(See \cite{Me1,Me2,MM} for the definition and properties of quasicrystals.)
A set $\Lambda\subset \mathbb{R}^d$ is called a {\it universal sampling set} \cite{OU2,MM} if
$\mathcal{D}(\Lambda)$ exists and $\Lambda$ is a set of stable sampling
for any compact set $K\subset \mathbb{R}^d$
with $|K|<\mathcal{D}(\Lambda)$, which means, in the terminology used in this paper,
that $\delta_\Lambda$ is an $\mathcal{F}$-measure
for $L^2(K)$ if $|K|<\mathcal{D}(\Lambda)$. A universal sampling set $\Lambda$
will thus yield a frame for $L^2(K)$ where $K=\cup_{i=1}^N\,[x_i-\epsilon/2,x_i+\epsilon/2]$
for any $x_1,\dots,x_N\in \mathbb{R}$ if $\epsilon>0$ is small enough
and dependent on the $x_i$'s. However, the associated frame constants are dependent on the
points $x_i$'s as well and it might not be possible to find frame bounds compatible
with all the finite  subsets $X=\{x_i,\,\,i=1,\dots N\}$ of real numbers.
In fact, this will be the case for simple quasi-crystals
as the following result shows. The proof is based on an idea used by B.~Matei
to show us that simple quasicrystals cannot yield  frames for $L^2(F)$
if $F$ is unbounded (\cite{Ma}).
\begin{proposition}
Let $\Lambda\subset \mathbb{R}$ be a simple quasicrystal. Then,  for any $B>0$,
the measure $\delta_\Lambda$ cannot belong to
$\mathcal{B}(\mathbb{R}, B)$.
\begin{proof} The main ingredient of this proof is that any simple quasicrystal $\Lambda$
is an harmonious set (see \cite{Me1,Me2}), which implies the existence of a
sequence $\{x_j\}_{j\ge 1}$ of real numbers having the property that
$$
\sup_{\lambda \in \Lambda}\,|e^{-2\pi i \lambda x_j}-1|\to 0\quad \text{as}\,\,j\to \infty.
$$
Given any integer $N\ge 1$ and any $\epsilon>0$, we can thus find some elements
$y_j=x_{n_j}$, $j=1,\dots,N$, of our sequence such that
$$
\sup_{\lambda \in \Lambda}\,|e^{-2\pi i \lambda y_j}-1|\le \epsilon.
$$
By Minkowski's inequality, if $t\in \mathbb{R}$, $R>0$ and $c_1,\dots,c_N\in \mathbb{C}$,
 we have, letting $\mu=\delta_\Lambda$, that
\begin{align*}
&\left(|\sum_{j=1}^N\,c_j|^2\right)^{1/2}
\,\left(\frac{\mu([t,t+R])}{R}\right)^{1/2}
=\left(\frac{1}{R}\,\int_{[t,t+R]}\,|\sum_{j=1}^N\,c_j|^2\,d\mu(\lambda)\right)^{1/2}\\
&\le \left(\frac{1}{R}\,\int_{[t,t+R]}\,|\sum_{j=1}^N\,c_j\,
(1-e^{-2\pi i \lambda y_j})|^2\,d\mu(\lambda)\right)^{1/2}
+\left(\frac{1}{R}\,\int_{[t,t+R]}\,|\sum_{j=1}^N\,c_j\,
e^{-2\pi i \lambda y_j}|^2\,d\mu(\lambda)\right)^{1/2}\\
&\le \epsilon\,\left(\frac{1}{R}\,\int_{[t,t+R]}\,\left(\sum_{j=1}^N\,|c_j|
\right)^2\,d\mu(\lambda)\right)^{1/2}
+\left(\frac{1}{R}\,\int_{[t,t+R]}\,|\sum_{j=1}^N\,c_j\,
e^{-2\pi i \lambda y_j}|^2\,d\mu(\lambda)\right)^{1/2}\\
\end{align*}
Choosing $c_j=1$ for all $j$ and taking limits as $R\to \infty$, we obtain that
$$
N\,\mathcal{D}(\mu)^{1/2}\le \epsilon\,N\,\mathcal{D}(\mu)^{1/2}+
\limsup_{R\to \infty\,}\left(\frac{1}{R}\,\int_{[t,t+R]}\,|\sum_{j=1}^N\,
e^{-2\pi i \lambda y_j}|^2\,d\mu(\lambda)\right)^{1/2}
$$
If $\mu\in \mathcal{B}(\mathbb{R}, B)$, it would then follow
from part (a) of Theorem \ref{main2}
 that, for any $\epsilon>0$
and any $N\ge 1$,
$$
N\,(1-\epsilon)\,\mathcal{D}(\mu)^{1/2}\le B\, N^{1/2}
$$
which yields a contradiction since $\mathcal{D}(\mu)=\mathcal{D}(\Lambda)>0$
if $\Lambda$ is a simple quasicrystal.
\end{proof}
\end{proposition}
Despite the previous negative result, we will to show the existence
of a discrete set $\Lambda$ such that the measure $\delta_\Lambda\in
\mathcal{B}(\mathbb{R},A, B)$ for some constants $0<A<B<\infty$.
In doing so, our task will be greatly simplified by
the following powerful recent result of S. Nitzan, A. Olevskii and A. Ulanovskii.
\begin{theorem}[\cite{NOU}]\label{NOU}
Every measurable set $E\subset \mathbb{R}$ with $|E|<\infty$
admits a
discrete set
$\Lambda$ such that $\{e^{2\pi i \lambda x}\}_{\lambda\in
\Lambda}$ is a frame for $L^2(E)$.
\end{theorem}
The proof of this last theorem  itself is far from trivial if $E$ is an unbounded set. It
uses a result
 in \cite{MSS} in which the authors solve the long standing
Kadison-Singer conjecture.
The idea to reach our goal  is then to use Theorem \ref{NOU} for a particular unbounded set $E$
satisfying the properties in the following lemma.
\begin{lemma}\label{E-lem}
There exists a set an open set $E\subset \mathbb{R}$
with $|E|<\infty$
such that, for any $x_1,\dots,x_m \in \mathbb{R}$,
 we have
$$
E\cap (E+x_1)\cap\dots(E+x_m)\ne\emptyset
$$
\end{lemma}
\begin{proof} Define for any integer $j\ge 0$ the set
$$
E_j=\cup_{k\in \mathbb{Z}}\,(k/2^j-1/2^{j+|k|+1},
k/2^j+1/2^{j+|k|+1})
$$
and let
$$
E=\cup_{j\ge 0}\, E_j.
$$
Note that $|E_j|=\sum_{k\in \mathbb{Z}}\,1/2^{j+|k|}=3/2^j$
and thus
$
|E|\le 3\,\sum_{j=0}^\infty\,2^{-j}=6<\infty.
$
Let $x_1,\dots,x_m\in \mathbb{R}$. Note that
$$
E\cap (E+x_1)\ne \emptyset
\iff x_1\in E-E.
$$
Since $(-1/2,1/2)\subset E$ and $E$ contains a neighborhood
of each integer, we have $E-E=\mathbb{R}$, so our claim
is true for $m=1$. To prove our claim for arbitrary $m$,
we use an induction argument. If we have
$$
E\cap (E+x_1)\cap\dots(E+x_{m-1})\ne\emptyset,
$$
let $J$ be a non-empty open interval contained in the intersection above.
The intersection $(E+x_m)\cap J$ is non-empty
if and only if $x_m\in J-E$. If $j\ge 0$ is chosen so that
$1/2^{j}<|J|$, we have $J-E_j=\mathbb{R}$
since $E_j$ contains a neighborhood of each point in $2^{-j}\,\mathbb{Z}$. Hence, $J-E=\mathbb{R}$ and our claim follows.
\end{proof}
\begin{theorem}\label{th27}
There exists a
discrete set
$\Lambda$ such that $\delta_\Lambda\in
\mathcal{B}(\mathbb{R},A, B)$ for some constants $0<A<B<\infty$.
\end{theorem}
\begin{proof}
let $E$ be the set constructed in the previous lemma
 and let $\Lambda\subset \mathbb{R}$ be a discrete set such that
$\delta_\Lambda$ is an $\mathcal{F}$-measure for $L^2(E)$
and whose existence follows from Theorem  \ref{NOU}.
We claim that, for any
$x_1,\dots,x_m\in \mathbb{R}$, there exists $\epsilon>0$
such that the collection
$\{e^{2\pi i \lambda x}\}_{\lambda\in
\Lambda}$ is a frame for $L^2(\cup_{i=1}^m\,x_i+Q_\epsilon)$
with frame bounds independent of $x_1,\dots, x_m$ and $\epsilon$.
Indeed, if $E$ is as above, there exists, by Lemma \ref{E-lem},
 $z\in E\cap (E+x_1)\cap\dots(E+x_m)\ne\emptyset$.
Then $z=z_1+x_1=\dots=z_m+x_m$  with $z_i\in E$
or $x_i-z=-z_i\in -E=E$, $i=1,\dots,m$.
For $\epsilon>0$ small enough, we have
$\cup_{i=1}^m\, (x_i-z)+Q_\epsilon\subset E$
and $\{e^{2\pi i \lambda x}\}_{\lambda\in
\Lambda}$ is a frame for $L^2(\cup_{i=1}^m\,(x_i-z)+Q_\epsilon)$
and thus also for $L^2(\cup_{i=1}^m\,x_i+Q_\epsilon)$.
\end{proof}

We note that the constants $A,B$ found in Theorem \ref{th27} are 
not equal since ${\mathcal E}(\Lambda)$ is a Fourier frame for the unbounded 
set $E$ with $A, B$ being the frame bounds. 
It was shown in \cite{[GL1]} that such set cannot admit any tight Fourier frames, 
so $A,B$ cannot be equal. It is unknown whether we can make $A,B$ equal for some other discrete set
$\Lambda$.

\end{document}